\documentclass[twoside,11pt]{article}

\usepackage[top=1in,bottom=1in,left=1in,right=1in]{geometry}
\usepackage[sort&compress]{natbib} 
\usepackage{amsfonts, amsmath, amssymb, amsthm, constants, bbm}
\usepackage{caption}
\usepackage{subcaption}

\usepackage{cases}
\usepackage{color}
\definecolor{darkred}{RGB}{100,0,0}
\definecolor{darkgreen}{RGB}{0,100,0}
\definecolor{darkblue}{RGB}{0,0,150}

\usepackage{hyperref}
\hypersetup{colorlinks=true, linkcolor=darkred, citecolor=darkgreen, urlcolor=darkblue}
\usepackage{url}

\newtheorem{thm}{Theorem}
\newtheorem{prp}{Proposition}
\newtheorem{lem}{Lemma}

\def\beq{\begin{equation}}
\def\eeq{\end{equation}}
\def\beqn{\begin{eqnarray*}}
\def\eeqn{\end{eqnarray*}}
\def\bitem{\begin{itemize}}
\def\eitem{\end{itemize}}
\def\benum{\begin{enumerate}}
\def\eenum{\end{enumerate}}
\def\bmult{\begin{multline*}}
\def\emult{\end{multline*}}
\def\bcenter{\begin{center}}
\def\ecenter{\end{center}}

\newcommand{\thmref}[1]{Theorem~\ref{thm:#1}}
\newcommand{\prpref}[1]{Proposition~\ref{prp:#1}}

\newcommand{\lemref}[1]{Lemma~\ref{lem:#1}}
\newcommand{\secref}[1]{Section~\ref{sec:#1}}

\DeclareMathOperator{\trace}{Trace}

\def\cE{\mathcal{E}}

\def\cG{\mathcal{G}}

\def\cV{\mathcal{V}}

\def\bW{\boldsymbol{A}}
\def\bB{\boldsymbol{B}}

\def\bU{\boldsymbol{U}}

\def\bW{\boldsymbol{W}}

\def\bZ{\boldsymbol{Z}}

\def\bbG{\mathbb{G}}

\def\bbR{\mathbb{R}}

\newcommand{\E}{\operatorname{\mathbb{E}}}
\renewcommand{\P}{\operatorname{\mathbb{P}}}
\newcommand{\Var}{\operatorname{Var}}

\newcommand{\pr}[1]{\mathbb{P}\left(#1\right)}

\def\iid{\stackrel{\rm iid}{\sim}}
\def\Bin{\text{Bin}}

\def\eps{\varepsilon}

\usepackage{dsfont}
\newcommand{\1}{\mathds{1}}

\def\tot{W}
\def\scan{W_{[n]}^*}
\def\nn{n^{(2)}}
\def\NN{N^{(2)}}
\def\kk{k^{(2)}}

\def\cheb{R}

\def\Lt{\tilde{L}}
\def\kmin{k_{\rm min}}

\newcommand{\wh}[1]{\widehat{#1}}
\newcommand{\IND}[1]{\mathbbm{1}_{\{ #1 \}}}

\definecolor{purple}{rgb}{0.4,.1,.9}

\begin{document}
\thispagestyle{empty}

\noindent {\sc \LARGE Community Detection in Random Networks}

\bigskip 
\noindent {\Large
Ery Arias-Castro\footnote{Department of Mathematics, University of California, San Diego, USA} and Nicolas Verzelen\footnote{INRA, UMR 729 MISTEA, F-34060 Montpellier, FRANCE}} \\[.05in]

\bigskip
\noindent
We formalize the problem of detecting a community in a network into testing whether in a given (random) graph there is a subgraph that is unusually dense.  We observe an undirected and unweighted graph on $N$ nodes.  Under the null hypothesis, the graph is a realization of an Erd\"os-R\'enyi graph with probability $p_0$.  Under the (composite) alternative, there is a subgraph of $n$ nodes where the probability of connection is $p_1 > p_0$.  We derive a detection lower bound for detecting such a subgraph in terms of $N, n, p_0, p_1$ and exhibit a test that achieves that lower bound.  We do this both when $p_0$ is known and unknown.  We also consider the problem of testing in polynomial-time.  As an aside, we consider the problem of detecting a clique, which is intimately related to the planted clique problem.  Our focus in this paper is in the quasi-normal regime where $n p_0$ is either bounded away from zero, or tends to zero slowly.

\medskip


\medskip
\bcenter
{\large \em Dedicated to the memory of Yuri I.~Ingster}
\ecenter

\section{Introduction}
\label{sec:intro}

In recent years, the problem of detecting communities in networks has received a large amount of attention, with important applications in the social and biological sciences, among others \citep{Santo201075}.  The vast majority of this expansive literature focuses on developing realistic models of (random) networks \citep{RevModPhys.74.47,Barabasi15101999}, on designing  methods for extracting communities from such networks \citep{Newman06062006,PhysRevE.74.016110,Girvan11062002} and on fitting  models to network data \citep{MR2906868}.  

The underlying model is that of graph $\cG = (\cE, \cV)$, where $\cE$ is the set of edges and $\cV$ is the set of nodes.  For example, in a social network, a node would represent an individual and an edge between two nodes would symbolize a friendship or kinship of some sort shared by these two individuals.  In the literature just mentioned, almost all the methodology has concentrated on devising graph partitioning methods, with the end goal of clustering the nodes in $\cV$ into groups with strong inner-connectivity and weak inter-connectivity \citep{PhysRevE.80.056117,PhysRevE.69.026113,Bickel15122009}.

In this euphoria, perhaps the most basic problem of actually detecting the {\em presence} of a community in an otherwise homogeneous network has been overlooked.  From a practical standpoint, this sort of problem could arise in a dynamic setting where a network is growing over time and monitored for clustering.  From a mathematical perspective, probing the limits of detection (i.e., hypothesis testing) often offers insight into what is possible in terms of extraction (i.e., estimation).  

Many existing community extraction methods can be turned into community detection procedures.  For example, one could decide that a community is present in the network if the modularity of \cite{PhysRevE.69.026113} exceeds a given threshold.  To set this threshold, one needs to define a null model.  \cite{PhysRevE.69.026113} implicitly assume a random graph conditional on the node degrees.  Here, we make the simplest assumption that the null model is an Erd\"os-R\'enyi random graph \citep{MR1864966}. 

In this context, we also touch on another line of work, that of detecting a clique in a random graph --- the so-called Planted (or Hidden) Clique Problem \citep{MR2735341,MR1662795,dekel}.  Although the emphasis there is to find the detection performance of computationally tractable algorithms, we mostly ignore computational consideration and simply establish the absolute detection limits of any algorithm whatsoever.

\subsection{The framework}
\label{sec:framework}

We address a stylized community detection problem, where the task is to detect the presence of clustering in the network and is formalized as a hypothesis testing problem.  We observe an {\em undirected} graph $\cG = (\cE, \cV)$ with $N := |\cV|$ nodes.  Without loss of generality, we take $\cV = [N] := \{1, \dots, N\}$.  The corresponding adjacency matrix is denoted $\bW \in \{0,1\}^{N \times N}$, where $W_{i,j} = 1$ if, and only if, $(i,j) \in \cE$, meaning there is an edge between nodes $i, j \in \cV$.  Note that $\bW$ is symmetric, and we assume that $W_{ii} = 0$ for all $i$.  Under the null hypothesis, the graph $\cG$ is a realization of $\bbG(N, p_0)$, the Erd\"os-R\'enyi random graph on $N$ nodes with probability of connection $p_0 \in (0,1)$; equivalently, the upper diagonal entries of $\bW$ are independent and identically distributed with $\P(W_{i,j} = 1) = p_0$ for any $i \neq j$.  Under the alternative, there is a subset of nodes indexed by $S \subset    \cV$ such that $\P(W_{i,j} = 1) = p_1$ for any 
$i,j \in S$ with $i \neq j$, with everything else the same.  We assume that $p_1 > p_0$, implying that the connectivity is stronger between nodes in $S$.  When $p_1 = 1$, the subgraph with node set $S$ is a clique.
The subset $S$ is not known, although in most of the paper we assume that its size $n := |S|$ is known.  

We study detectability in this framework in asymptotic regimes where $n, N \to \infty$, and $p_0, p_1$ may also change; all these parameters are assumed to be functions of $N$.  A test $T$ is a function that takes $\bW$ as input and returns $T =1$ to claim there is a community in the network, and $T=0$ otherwise.  The (worst-case) risk of a test $T$ is defined as
\[
\gamma_N(T) = \P_0(T = 1) + \max_{|S| = n} \P_S(T = 0),
\]
where $\P_0$ is the distribution under the null and $\P_S$ is the distribution under the alternative where $S$ indexes the community.  We say that a sequence of tests $(T_N)$ for a sequence of problems $(\bW_N)$ is asymptotically powerful (resp.~powerless) if $\gamma_N(T_N) \to 0$ (resp.~$\to 1$).  Practically speaking, a sequence of tests is asymptotically powerless if it does not perform substantially better than any guessing that ignores the adjacency matrix $\bW$.  We will often speak of a test being powerful or powerless when in fact referring to a sequence of tests and its asymptotic power properties.

\subsection{Closely related work}

We take the beaten path, following the standard approach in statistics for analyzing such composite hypothesis testing problems, in particular, the work of \cite{MR1456646} and others \citep{MR2065195,MR1976749,MR2662357} on the detection of a sparse (normal) mean vector.  Most closely related to our work is that of \cite{1109.0898}.  Specializing their results to our setting, they derive lower bounds and upper bounds for the same detection problem when the graph is directed and the probability of connection under the null (denoted $p_0$) is fixed, which is a situation where the graph is extremely dense.  Their work leaves out the interesting regime where $p_0 \to 0$, which leads to a null model that is much more sparse.

\subsection{Main Contribution}
 
Our main contribution in this paper is to derive a sharp detection boundary for the problem of detecting a community in a network as described above.
We focus here on the quasi-normal regime\footnote{The quasi-Poisson regime where $n p_0 \to 0$ polynomially fast is qualitatively different and necessitates different proof arguments.  This is beyond the scope of this paper and will appear somewhere else.} 
where $n p_0$ is either bounded away from zero, or tends to zero slowly, specifically,
\beq \label{n-p0}
\log\left(1 \vee \frac{1}{np_0}\right)=o\left[\log\left(\frac{N}{n}\right)\right]\ .
\eeq

On the one hand, we derive an information theoretic bound that applies to all tests, meaning conditions under which all tests are powerless.  On the other hand, we display a test that basically achieves the best performance possible.   
The test is the combination of the two natural tests that arise in \cite{1109.0898} and much of the work in that field \citep{ingster2010detection,anova-hc}:
\bitem
\item {\em Total degree test.}  This test rejects when the total number of edges is unusually large.  This is global in nature in that it cannot be directly turned into a method for extraction.  
\item {\em Scan (or maximum modularity) test.}  This test amounts to turning modularity into a test statistic by rejecting when its maximum value is unusually large.  It is strictly speaking the generalized likelihood ratio test under our framework.
\eitem

We also consider the situation, common in practice, where $p_0$ is unknown. Interestingly, the detection boundary becomes larger than in the former setting when $n$ is moderately sparse. We derive the corresponding lower bound in this situation and design a test that achieves this bound.
The test is again the combination of the two tests:
\bitem
\item {\em Degree variance test.}  This test is based on the differences between two estimates for the degree variance, an analysis of variance of sorts.  (Note that the total degree test cannot be calibrated without knowledge of $p_0$.)
\item {\em Scan test.}  This test can be calibrated in various ways when $p_0$ is unknown, for example by estimation of $p_0$ based on the whole graph, or by permutation.  We study the former.
\eitem

Finally, we consider various polynomial-time algorithms, the main one being a convex relaxation of the scan test based on a sparse eigenvalue problem formulation.  Our inspiration there comes from the recent work of \cite{berthet}. We discuss the discrepancy between the performances of the scan test and the relaxed scan test and compare it with other polynomial-time tests.

We summarize our findings in Tables~\ref{tab:main} and~\ref{tab:poly}, where 
\[R = \frac{\sqrt{n} (p_1 - p_0)}{\sqrt{p_0 (1-p_0)}}\]
is (up to $\sqrt{n/2}$ factor) the SNR for detecting the dense subgraph when it is known.

\def\arraystretch{2}
\begin{table}
\caption{Detection boundary and near-optimal algorithms. For any sequence $a$ and $b$ going to infinity,  $a\lll b$ (resp. $a\ggg b$) means that there exists $\epsilon>0$ arbitrarily small such that $a\leq b^{1-\epsilon}$ (resp. $a\geq b^{1+\epsilon}$) }
\label{tab:main}
\centering
\centering
{\small 
\begin{tabular}{c ||c | c|| c |c }
  & \multicolumn{2}{|c||}{\large $p_0$ known } & \multicolumn{2}{c}{\large $p_0$ unknown} \\ \cline{2-5}
  & $ n \lll N^{2/3}$ & $n \ggg N^{2/3}$ & $n \lll N^{3/4}$ & $n \ggg N^{3/4}$ \\
\hline\hline 
$p_0\gg \frac{\log(N/n)}{n}$ & $R > 2 \sqrt{\log(N/n)}$ & $R > N/n^{3/2} $ &$R > 2 \sqrt{\log(N/n)}$ & $R > N^{3/4}/n$ \\
$p_0\ll \frac{\log(N/n)}{n}$ & $R >  \frac{2\log(N/n)}{\sqrt{np_0}\log\left(\frac{\log(N/n)}{np_0}\right)}$ & $R > N/n^{3/2} $ & $R > \frac{ 2\log(N/n)}{\sqrt{np_0}\log\left(\frac{\log(N/n)}{np_0}\right)}$ & $R > N^{3/4}/n$ \\
\hline
&{\sc Scan test} & {\sc Tot. Deg. test}&{\sc Scan test} & {\sc Deg. Var. test}
\end{tabular}
}

\end{table}

\begin{table}
\caption{Polynomial time algorithms}
\label{tab:poly}
\centering
{\small 
\begin{tabular}{c | c|| c |c }
 \multicolumn{2}{c||}{\large $p_0$ known } & \multicolumn{2}{c}{\large $p_0$ unknown} \\ \hline 
$n \lll \sqrt{N}$ & $n \ggg  \sqrt{N}$ & $n \lll \sqrt{N}$ & $n \ggg \sqrt{N}$ \\
\hline\hline 
$R > 2 \sqrt{N \log N}$ & $R > N/n^{3/2} $ & $R > 2 \sqrt{N \log N}$ & $R > N^{3/4}/n$ \\
\hline
{\sc Relax. Scan test} & {\sc Tot. Deg. test}& {\sc Relax. Scan test} & {\sc Deg. Var. test}
\end{tabular}
}
\end{table}

\subsection{Finding a clique}
We start the paper by addressing the problem of detecting the presence of a large clique in the graph, and treat it separately, as it is an interesting case in its own right.  It is simpler and allows us to focus on the regime where $n/\log N \to \infty$ in the rest of the paper.
We establish a lower bound and prove that the following (obvious) test achieves that bound:
\bitem
\item {\em Clique number test.}  This tests rejects when the size of the clique number of the graph is unusually large.  It can be calibrated without knowledge of $p_0$, for example by permutation, but we do not know of a polynomial-time algorithm that comes even close.
\eitem

\subsection{Content}

In \secref{clique}, we consider the problem of detecting the presence of a large clique and analyze the clique number test.  
In \secref{subgraph}, we consider the more general problem of detecting a densely connected subgraph and analyze the total degree test and the scan test.  The more realistic situation of unknown $p_0$ is handled in \secref{adapt}. In \secref{dense}, we investigate polynomial-time tests.
We then discuss our results and the outlook in \secref{discussion}.  
The technical proofs are postponed to \secref{proofs}.

\subsection{General assumptions and notation}
\label{sec:notation}

We assume throughout that $N \to \infty$ and the other parameters $n, p_0, p_1$ (and more) are allowed to change with $N$, unless specified otherwise.  This dependency is left implicit.  In particular, we assume that $n/N \to 0$, emphasizing the situation where the community to be detected is small compared to the size of the whole network.  (When $n$ is of the same order as $N$, the total degree test is basically optimal.)  We assume that $p_0$ is bounded away from 1, which is the most interesting case by far, and that $N^2 p_0 \to \infty$, the latter implying that the number of edges in the network (under the null) is not bounded.  We also hypothesize that either $p_1 = 1$ or $n \to \infty$ with $n^2 p_1 \to \infty$, there is a non-vanishing chance that the community does not contain any edges, precluding any test to be powerful.  

We use standard notation such as $a_n \sim b_n$ when $a_n/b_n \to 1$; $a_n = o(b_n)$ when $a_n/b_n \to 0$; $a_n = O(b_n)$ when $a_n/b_n$ is bounded; $a_n \asymp b_n$ when $a_n = O(b_n)$ and $b_n = O(a_n)$; $a_n\prec b_n$ when there exists a positive constant $C$ such that $a_n\leq Cb_n$ and   $a_n\succ b_n$ when there exists a positive constant $C$ such that $a_n\geq Cb_n$.
For an integer $n$ let $\nn = n(n-1)/2$.  For two distributions $L_1 $ and $L_2$ on the real line, let $L_1 * L_2$ denote their convolution, which is the distribution of the sum two independent random variables $X_1 \sim L_1$ and $X_2 \sim L_2$.  

Because of its importance in describing the tails of the binomial distribution, the following function --- which is the relative entropy or Kullback-Leibler divergence of ${\rm Bern}(q)$ to ${\rm Bern}(p)$ --- will appear in our results:
\beq \label{H}
H_p(q) = q \log \left(\frac{q}p\right) + (1-q) \log \left(\frac{1 -q}{1 -p}\right), \quad p, q \in (0,1).
\eeq

\section{Detecting a large clique in a random graph}
\label{sec:clique}

We start with specializing the setting to that of detecting a large clique, meaning we consider the special case where $p_1 = 1$.  In this section, $n$ is not necessarily increasing with $N$.

\subsection{Lower bound}
\label{sec:clique-lb}

We establish the detection boundary, giving sufficient conditions for the problem to be too hard for any test, meaning that all tests are asymptotically powerless.

\begin{thm} \label{thm:clique}
All tests are asymptotically powerless if 
\beq \label{clique1}
{N \choose n} p_0^{\frac{n (n-1)}2} \to \infty.
\eeq
\end{thm}

The result is, in fact, very intuitive.  Condition \eqref{clique1} implies that, with high probability under the null, the clique number is at least $n$, which is the size of the implanted clique under the alternative.  This is a classical result in random graph theory, and finer results are known --- see \citep[Chap.~11]{MR1864966}.  The arguments underlying \thmref{clique} are, however, based on studying the likelihood ratio test when a uniform prior is assumed on the implanted clique $S$, which is the standard approach in detection settings; see \citep[Ch.~8]{TSH}.  In this specific setting, the second moment method --- which consists in showing that the variance of the likelihood ratio tends to 0 --- suffices.

\subsection{The clique number test}
\label{sec:clique-ub}

Computational considerations aside, the most natural test for detecting the presence of a clique is the clique number test defined in the Introduction.  We obtain the following.

\begin{prp} \label{prp:clique}
The clique number test is powerful if 
\beq \label{clique2}
{N \choose n} p_0^{\frac{n (n-1)}2} \to 0.
\eeq
\end{prp}

The proof is entirely based on the fact that, when \eqref{clique2} holds, the clique number under the null is at most $n-1$ with high probability \citep[Th.~11.6]{MR1864966}, while it is at least $n$ under the alternative.  (Thus the proof is omitted.)  We conclude that the clique number test is seen to achieve the detection boundary established in \thmref{clique}.

\section{Detecting a dense subgraph in a random graph}
\label{sec:subgraph}

We now consider the more general setting of detecting a dense subgraph in a random graph.  We start with an information bound that applies to all tests, regardless of their computational requirements.  We then study the total degree test and the scan test, showing that the test that combines them with a simple Bonferroni correction is essentially optimal.  

\subsection{Lower bound}
\label{sec:subgraph-lb}

When assuming infinite computational power, what is left is the purely statistical challenge of detecting the subgraph.  For simplicity, we assume that $n$ is not too small, specifically,
\beq \label{n-log}
\frac{n}{\log N} \to \infty,
\eeq
though our result below partially extends to this, particularly when $p_1$ is constant.   
As usual, a minimax lower bound is derived by choosing a prior over the composite alternative.  Assuming that $p_0$ and $p_1$ are known, because of symmetry, the uniform prior over the community $S$ is least favorable, so that we consider testing
\beq \label{model}
H_0: \cG \sim \bbG(N,p_0) \text{ versus } \bar{H}_1: \cG \sim \bbG(N,p_0; n,p_1),
\eeq
where the latter is the model where the community $S$ is chosen uniformly at random among subset of nodes of size $n$, and then for $i \ne j$, $\P(W_{i,j} = 1) = p_1$ if $i,j \in S$,  while $\P(W_{i,j} = 1) = p_0$ otherwise.
For this simple versus simple testing problem, the likelihood ratio test is optimal, which is what we examine to derive the following lower bound.  
Remember the entropy function defined in \eqref{H}.

\begin{thm} \label{thm:lower}
Assuming \eqref{n-log} and \eqref{n-p0} hold, all tests are asymptotically powerless if 
\beq \label{lower1}
\frac{p_1 - p_0}{\sqrt{p_0}} \, \frac{n^2}N \to 0,
\eeq
and
\beq \label{lower2}
\limsup\frac{n H(p_1)}{2 \log (N/n)} < 1.
\eeq
\end{thm}

Conditions \eqref{lower1} and \eqref{lower2} have their equivalent in the work of \cite{1109.0898}.   That said, \eqref{lower2} is more complex here because of the different behaviors of the entropy function according to whether $p_1/p_0$ is small or large --- corresponding to the difference between large deviations and moderate deviations of the binomial distribution.  Only in the case where $p_1/p_0 \to 1$ is the normal approximation to the binomial in effect.  

To better appreciate \eqref{lower2}, note that it is equivalent to 
\beq \label{lower2a}
\limsup \frac{(p_1 - p_0)^2}{4 p_0 (1-p_0)} \, \frac{n}{\log (N/n)} < 1, \quad \text{ when } \frac{n p_0}{\log (N/n)} \to \infty;
\eeq
and
\beq \label{lower2b}
\limsup \frac{p_1}{2 (1-p_0)} \, \frac{n}{\log (N/n)} \log\left(\frac{\log (N/n)}{n p_0}\right) < 1, \quad \text{ when } \frac{n p_0}{\log (N/n)} \to 0.
\eeq
In \eqref{lower2a}, $n p_0$ is larger and only the moderate deviations of the binomial distribution are involved, while in \eqref{lower2b}, $n p_0$ is smaller and the large deviations come into play.  

\thmref{lower} happens to be sharp because, as we show next, the test that combines the total degree test and the scan test is asymptotically powerful when the conditions \eqref{lower1} and \eqref{lower2} are --- roughly speaking --- reversed.

\subsection{The total degree test}
\label{sec:total}

The total degree test rejects for large values of 
\beq \label{total-stat}
\tot := \sum_{1 \leq i < j \leq N} W_{i,j}.
\eeq
The resulting test is exceedingly simple to analyze, since  
\beq \label{total-dist}
\tot \sim \Bin(\NN - \nn, p_0) \ * \ \Bin(\nn, p_1). 
\eeq

\begin{prp} \label{prp:total}
The total degree tests is powerful if 
\beq \label{total}
\frac{p_1 - p_0}{\sqrt{p_0}} \, \frac{n^2}N \to \infty. 
\eeq
\end{prp}
It is equally straightforward to show that the total degree has risk strictly less than one --- meaning has some non-negligible power --- when the same ratio tends to a positive and finite constant, while it is asymptotically powerless when that ratio tends to zero.

\subsection{The scan test}
\label{sec:scan}

The scan test is another name for the generalized likelihood ratio test, and corresponds to the test that is based on the maximum modularity.  It is particularly simple when $p_0$ is known, as it rejects for large values of
\beq \label{scan-stat}
\scan := \max_{|S| = n} W_S, \qquad W_S := \sum_{i,j\in S, i < j} W_{i,j}.
\eeq

Unlike the total degree \eqref{total-stat}, the scan statistic \eqref{scan-stat} has an intricate distribution as the partial sums $W_S$ are not independent.  Nevertheless, the union bound and standard tail bounds for the binomial distribution lead to the following result.

\begin{prp} \label{prp:scan}
The scan test is powerful if 
\beq \label{scan}
\liminf \frac{n H(p_1)}{2 \log (N/n)} > 1.
\eeq
\end{prp}

\subsection{The combined test}
\label{sec:combined}

Having studied these two tests individually, we are now in a position to consider them together, by which we mean a simple Bonferroni combination which rejects when either of the two tests rejects.  Looking back at our lower bound and the performance bounds we established for these tests, we come to the following conclusion.  When the limit in \eqref{lower1} is infinite --- yielding \eqref{total} --- then the total degree test is asymptotically powerful by \prpref{total}.  When the limit inferior in \eqref{lower2} exceeds one --- yielding \eqref{scan} --- then the scan test is asymptotically powerful by \prpref{scan}.

\subsection{Adaptation to unknown $n$}
\label{sec:n-adapt}

 The scan statistic in \eqref{scan-stat} requires knowledge of $n$.  When this is unknown, the common procedure is to combine the scan tests at all different sizes $n$ using a simple Bonferroni correction.  This is done in \citep{1109.0898}, with the conclusion that the resulting test is essentially as powerful as the individual tests.  It is straightforward to see that, here too, the tail bound used in the proof of \prpref{scan} allows for enough room to scan over all subgraphs of all sizes.

\section{When $p_0$ is unknown: the fixed expected total degree model}
\label{sec:adapt}

Although it leads to interesting mathematics, the setting where $p_0$ is known is, for the most part, impractical.  In this section, we evaluate how not knowing $p_0$ changes the difficulty of the problem.  In fact, it makes the problem strictly more difficult in the denser regime.

There are (at least) two ways of formalizing the situation where $p_0$ is unknown.
In the first option, we still consider the exact same hypothesis testing problem, but maximize the risk over relevant subsets of $p_0$'s and $p_1$'s, since now even the null hypothesis is composite.
In the second option --- which is the one we detail --- for a given pair of probabilities $0 < p'_0 \le p_1 < 1$, we consider testing 
\beq \label{fixed_degree}
H_0: \cG \sim \bbG(N,p_0) \text{ versus } \bar{H}'_1: \cG \sim \bbG(N, p'_0; n, p_1), \quad p_0 := p'_0+(p_1-p'_0)\frac{n^{(2)}}{N^{(2)}}.
\eeq
Note that, in this setting, we still assume that $p_0, p_1, n$ are known to the statistician.
By design, the graph has the same expected total degree under the null and under the alternative hypotheses, that is we have 
\[\mathbb{E}_0(W)= N^{(2)}p_0+n^{(2)}(1-p_0)= \mathbb{E}'_S(W), \ \forall S: |S|=n,\]
where $\P'_S$ and $\E'_S$ denote the probability distribution and corresponding expectation under the model where, for any $i\neq j$, $\P(W_{i,j}=1)=p_1$ if $i,j\in S$, while $\P(W_{i,j}=1)=p'_0$ otherwise. 

The risk of  a test $T$ for this problem is defined as
\[\gamma_N'(T)=\mathbb{P}_{0}(T=1)+ \max_{|S|=n}\mathbb{P}'_S(T=0)\ .\]
We say that the a sequence of tests $(T_N)$ is asymptotically powerful for the problem with fixed expected total degree (resp. powerless) if $\gamma_N'(T_N)\rightarrow 0$ (resp. $\gamma_N'(T_N)\rightarrow 1$).

We first compute the detection boundary for this problem and then exhibit some tests achieving this detection boundary. 
Interestingly, these tests do not require the knowledge of $p_0$ and $p_1$, or even $n$, so that they can be used in the original setting \eqref{model} when these parameters are unknown. 

\subsection{Lower bound}
\begin{thm} \label{thm:lower_unknownp_0}
Assuming \eqref{n-log} holds and that
\beq \label{n-p0_unknown}
\log\left(1 \vee \frac{1}{np'_0}\right)=o\left[\log\left(\frac{N}{n}\right)\right]\ ,
\eeq
all tests  are asymptotically powerless for the problem \eqref{fixed_degree} if 
\beq \label{lower1_eq_unknown}
\frac{p_1-p'_0}{\sqrt{p'_0}}\frac{n^{3/2}}{N^{3/4}}\rightarrow 0 
\eeq
and
\beq \label{lower2_eq_unknown}
\limsup\frac{n H_{p'_0}(p_1)}{2 \log (N/n)} < 1.
\eeq
\end{thm}

Comparing with \thmref{lower}, where $p_0$ is assumed to be known, 
the condition \eqref{lower1_eq_unknown} is substantially weaker than the corresponding condition \eqref{lower1}, while we shall see in the proof that \eqref{lower2_eq_unknown} is comparable to \eqref{lower2}.
That said, when $n^2<N$, the entropy condition \eqref{lower2} is a stronger requirement than either \eqref{lower1} or \eqref{lower1_eq_unknown}, implying that the setting where $p_0$ is known and the setting where unknown are asymptotically as difficult in that case.

\subsection{Degree variance test} 

By construction, the total degree $W$ has the same expectation under the null and under the alternative in the testing problem with fixed expected total degree --- and same variance also up to second order --- making it difficult to see how to fruitfully use this statistic in this context.  

We design instead a test based on comparing the two estimators for the node degree variance, not unlike an analysis of variance.
Let 
\beq \label{degree}
W_{i\cdot}=\sum_{j\neq i} W_{i,j}
\eeq
denote the degree of node $i$ in the whole network.
The first estimate is simply the maximum likelihood estimator under the null
\[
V_1 = (N-1) \frac{\NN}{\NN-1} \hat p_0 (1- \hat p_0), \quad \hat p_0 := \frac{W}{\NN}.
\]
The second estimator is some sort of sample variance, modified to account for the fact that the $W_{i\cdot}$ are not independent
\[
V_2 = \frac{1}{N-2}\sum_{i=1}^N\left(W_{i\cdot}-(N-1)\hat{p}_0\right)^2.
\]
Both estimators are unbiased for the degree variance under the null, meaning, $\E_0 V_1 = \E_0 V_2 = (N-1) p_0(1-p_0)$.
Under the alternative, $V_2$ tends to be larger than $V_1$, leading to a test that rejects for large values of
\beq \label{Vstar}
V^* := \frac{V}{\sqrt{N}\hat{p}_0}, \quad V := V_2 - V_1.
\eeq

\begin{prp}\label{prp:degre_variance} Assume that $p_0\succ 1/N$. The degree variance test is asymptotically powerful under fixed expected total degree if
\beq \label{degre_variance}
\frac{(p_1-p'_0)^2}{p'_0}\frac{n^3}{N^{3/2}}\rightarrow \infty
\eeq
\end{prp}

The test based on $V^*$ achieves the part \eqref{lower1_eq_unknown} of the detection boundary. 
We note that computing $V^*$ does not require knowledge of $p_0$, $p_1$ or $n$, and in fact, its calibration can be done without any knowledge of these parameters via a form of parametric bootstrap, as we do for the scan test below.

\subsection{The scan test}

When $p_0$ is not available a priori, we have at least three options:
\bitem \setlength{\itemsep}{0in}
\item {\em Estimate $p_0$.}
We replace $p_0$ with its maximum likelihood estimator under the null, i.e., $\hat{p}_0 = \tot/\NN$, and then compare the magnitude of the observed scan statistic \eqref{scan-stat} with what one would get under a random graph model with probability of connection equal to $\hat{p}_0$.  
\item {\em Generalized likelihood ratio test.}
We simply implement the actual generalized likelihood ratio test \citep{Kul}, which rejects for large values of 
\[
\max_{|S| = n} \left[ \nn h(\hat{p}_{1,S}) + (\NN - \nn) h(\hat{p}_{0,S}) - \NN h(\hat{p}_{0})\right]\ ,
\]
where $h(p) := p \log p + (1-p) \log (1-p)$, $\hat{p}_0$ as above, and
\[
\hat{p}_{1,S} := \frac{W_S}{\nn}, \quad \hat{p}_{0,S} := \frac{\tot - W_S}{\NN - \nn}\ ,
\] 
which are the maximum likelihood estimates of $p_1$ and $p_0$ for a given subset $S$.

\item {\em Calibration by permutation.}
We compare the observed value of the scan statistic to simulated values obtained by generating a random graph with either the same number of edges --- which leads to a calibration very similar to the first option --- or the same degree distribution --- which is the basis for in the modularity function of \cite{PhysRevE.69.026113}.  
\eitem

We focus on the first option.
\begin{prp} \label{prp:scan_estimate}
Assume that $\lim \inf  p_0N^2/n>1$. The scan test calibrated by estimation of $p_0$ is asymptotically powerful for  fixed expected total degree  if 
\beq \label{scan_estimate1}
\liminf \frac{n H(p_1)}{2 \log (N/n)} > 1\ .
\eeq
\end{prp}
Hence, the scan test calibrated by estimation of $p_0$ achieves the entropy condition \eqref{lower2} without requiring the knowledge of $p_0$ or $p_1$.  We mention that adaptation to unknown $n$ may be achieved as described in \secref{n-adapt}.

\subsection{Combined test and full adaptation to unknown $p_0$}

A combination of the degree variance test and of the scan test calibrated by estimation of $p_0$ is seen to achieve the detection boundary established in \thmref{lower_unknownp_0}, without requiring knowledge of $p_0$ or $p_1$, or even $n$.

\section{Testing in polynomial-time}
\label{sec:poly}

While computing the total degree \eqref{total-stat} or the degree variance statistic \eqref{Vstar} can be done in linear time in the size of the network, i.e., in $O(N^2)$ time, computing the scan statistic \eqref{scan-stat} is combinatorial in nature and there is no known polynomial-time algorithm to compute it.
To see this, note that the ability to compute \eqref{scan-stat} in polynomial-time implies the ability to compute the size of the largest clique in the graph, since this is equal to
\[
\max\{n: \scan = \nn\}\ ,
\]
and computing the size of the largest clique in a general graph in known to be NP-hard \citep{MR0378476}, and even hard to approximate \citep{MR2277193}.

A question of particular importance in modern times is determining the tradeoff between statistical performance and computational complexity.
At the most basic level, this boils down to answering the following question: 
{\em What can be done in polynomial-time?}

\subsection{Convex relaxation scan test} \label{sec:relax}
We now suggest a convex relaxation to the problem of computing the scan statistic.
To do so, we follow the footsteps of \cite{berthet}, who consider the problem of detecting a sparse principal component based on a sample from a multivariate Gaussian distribution in dimension $N$.
Assuming the sparse component has at most $n$ nonzero entries, they show that a near-optimal procedure is based on the largest eigenvalue of any $n$-by-$n$ submatrix of the sample covariance matrix.
Computing this statistic is NP-hard, so they resort to the convex relaxation of \cite{aspremont}, which they also study.
We apply their procedure to $\bW^2$.

Formally, for a positive semidefinite matrix $\bB \in \bbR^{N \times N}$ and $1\le n \le N$, define 
\[
\lambda_n^{\rm max}(\bB) = \max_{|S| = n} \lambda^{\rm max}(\bB_S)\ ,
\]
where $\bB_S$ denotes the principal submatrix of $\bB$ indexed by $S \subset \{1, \dots, N\}$ and $\lambda^{\rm max}(\bB)$ the largest eigenvalue of $\bB$.
\cite{aspremont} relaxed this to
\[{\rm SDP}_n(\bB) = \max_{\bZ} \, \trace(\bB \bZ), \quad \text{ subject to } \bZ \succeq 0, \, \trace(\bZ) = 1, \, |\bZ|_1 \le n~,
\] 
where the maximum is over positive semidefinite matrices $\bZ = (Z_{st}) \in \bbR^{N \times N}$ and $|\bZ|_1 = \sum_{s,t} |Z_{st}|$.
We consider the relaxed scan test, which rejects for large values of
\beq \label{relaxed-stat}
{\rm SDP}_n(\bW^2)\ .
\eeq

When $p_0$ is known, we simply calibrate the procedure by Monte Carlo simulations, effectively generating $\bW_1, \dots, \bW_B$ i.i.d.~from $\bbG(N,p_0)$ and computing ${\rm SDP}_n(\bW_b^2)$ for each $b = 1, \dots, B$, and estimating the p-value by the fraction of $b$'s such that ${\rm SDP}_n(\bW_b^2) \ge {\rm SDP}_n(\bW^2)$.  Typically $B$ is a large number, and below we consider the asymptote where $B = \infty$.

When $p_0$ is unknown, we estimate $p_0$ as we did for the scan test in \prpref{scan_estimate}, and then calibrate the statistic by Monte Carlo, effectively using a form of parametric bootstrap.

In either case, we have the following.

\begin{prp} \label{prp:relaxed}
Assume that \eqref{n-p0} holds
and $n\leq  N^{1/2-t}$ for some $t>0$. 
Then, the relaxed scan test is powerful if 
\beq \label{relaxed}
\lim\inf \frac{n}{\sqrt{N\log(N)}}\frac{(p_1-p_0)^2}{p_0}> 2 \ .
\eeq
\end{prp}

To gain some insights on the relative performance of the scan test and the relaxed scan test, let us assume that $n^2\ll N$, and $n p_0 \gg \log(N/n)$.  Applying \prpref{scan} (or \prpref{scan_estimate}) in this setting, we find that the scan test is asymptotically powerful when
\[\frac{(p_1-p_0)^2}{p_0}\succ \frac{\log(N/n)}{n}\ .\]
Thus, comparing with \eqref{relaxed}, we lose a factor $\sqrt{N/\log(N)}$ when using the relaxed version. In the denser regime where $n^2\gg N \log(N)$, the total degree test and degree variance test both have stronger theoretical guarantees established in \prpref{total} and \prpref{degre_variance} respectively. 
Below we explain why the $\sqrt{N/\log(N)}$ loss is not unexpected.

\subsubsection*{Optimality}

The problem $H_0: \cG \sim \bbG(N,1/2) \text{ versus } H_1: \cG \sim \bbG(N,1/2; n,1)$ is called the Planted (or Hidden) Clique Problem \citep{MR2735341} and has become one of the most emblematic statistical problems where computational constraints seem to substantially affect the difficulty of the problem.  
Recent advances in compressed sensing and matrix completion have shown that computationally tractable algorithms can achieve the absolute information bounds (up to constants) in most cases.
In contrast, in the Planted (or Hidden) Clique Problem there is no known polynomial-time algorithm that can detect a clique of size $n = o(\sqrt{N})$ \citep{dekel}, while the clique test can detect a clique of size $n \asymp \log N$, as shown in Proposition \ref{prp:clique}.
In fact, the problem is provably hard in some computational models, such as monotone circuits~\citep{rossmanphd,feldman12}.  We refer to \cite{berthet} for a thorough discussion. 

More generally, we may want to characterize the sequences $(n,N,p_0,p_1)$ for which there are asymptotically powerful tests running in polynomial time.  In our findings, the only situation where we found this to be true was in the dense regime, where the total degree test is both powerful in the large-sample limit and computable in polynomial time.  (Replace this with the degree variance test when $p_0$ is unknown.)

\subsection{Other polynomial-time tests} \label{sec:dense}

\subsubsection{The maximum degree test}
Perhaps the first computationally-feasible test that comes to mind in the sparse regime is the test based on the maximum degree
\beq \label{max}
\max_{i=1,\dots,N} W_{i\cdot}\ , 
\eeq
where $W_{i\cdot}$ is the degree of node $i$ in the graph, defined in \eqref{degree}.

\begin{prp}\label{prp:max}
The maximal degree test is asymptotically powerful if $p_0\gg \log(N)/N$ and 
\[
\lim\inf\frac{n^2}{N\log(N)}\frac{(p_1-p_0)^2}{p_0(1-p_0)}>2\ .
\]
Under condition \eqref{n-p0}, the maximal degree test is asymptotically powerless if $\lim\sup \log(n)/\log(N) <1$ and 
\beq\label{eq:condition_lower_degree}
 \frac{n^2}{N\log(N)}\frac{(p_1-p_0)^2}{p_0(1-p_0)}\rightarrow 0
\eeq
\end{prp}

Comparing with  Propositions \ref{prp:total} and \ref{prp:relaxed}, we observe that the maximum degree test is either less powerful than the relaxed scan test (when $n\leq N^{1/2-t}$ for any $t>0$) or less powerful that the total degree test (when $n\gg \sqrt{N/\log(N)}$). For unknown $p_0$, the maximum degree test is also less powerful than the degree variance test.

\medskip

\subsubsection{Densest subgraph test}
Another possible avenue for designing computationally tractable tests for the problem at hand lies in algorithms for finding dense subgraphs of a given size. 
We follow \citep{khuller2009finding}, where the reader will find appropriate references and additional results.
Define the density of a subgraph $S \subset \cV$ as 
\[h(S) = \frac{|E_S|}{|S|}, \quad \text{ where } E_S = \{(i,j) \in S^2: W_{i,j} = 1\}\ .\]
Finding $S \subset \cV$ that maximizes $h(S)$ may be done in polynomial-time.

\begin{prp}\label{prp:densest_subgraph}
Assume that $p_0\gg \log(N)/N$. \begin{enumerate}
\item Under the null hypothesis,  $\max_S h(S)\sim_{\mathbb{P}_0} h(\cV)\sim N p_0/2$ and this maximum is achieved at subsets $S$ satisfying $|S|\sim N$.
\item The densest subgraph test is powerful if $\lim\inf \frac{np_1}{Np_0}>1$.
\item Assume that  $\frac{np_1}{Np_0}\rightarrow 0$. Under the alternative hypothesis,  $\max_S h(S)\sim_{\mathbb{P}_S} h(\cV)\sim_{\mathbb{P}_S} N p_0/2$ and this maximum is achieved at subsets $S$ satisfying $|S|\sim N$.
\end{enumerate}
\end{prp}

The condition $\lim\inf \frac{np_1}{Np_0}>1$ is stronger than what we have obtained for the relaxed scan test \eqref{relaxed} in the sparser case ($n\leq N^{1/2-t}$ for any $t>0$) and than what we have obtained for the total degree test \eqref{total} and the degree variance test \eqref{degre_variance} in the less sparse case $(n\gg \sqrt{N})$. If $np_1/Np_0\rightarrow 0$, then the densest subgraph statistic seems to behave like the total degree statistic and we therefore expect similar performances although we have no proof of this statement.

In order to improve the power, we would like to restrict our attention to subgraphs of size $n$ (assumed known for now) and use 
$\max_{|S| = n} h(S).$
Computing this, however, is NP-hard, and there is no known polynomial-time approximation within a constant factor.  
Nevertheless, the following variant statistic $\max_{|S| \ge n} h(S)$ can be approximated within a constant factor in polynomial-time. However, the power of the resulting test is not improved. 
Since the statistic $\max_{|S| \ge n} h(S)$ may only be approximated within a constant factor,  the resulting test is powerful only if $np_1\geq C N p_0$ where $C$ is positive constant that depends on this approximation factor.

\section{Discussion}
\label{sec:discussion}

With this paper, we have established the fundamental statistical (information theoretic) difficulty of detecting a community in a network, modeled as the detection of an unusually dense subgraph within an Erd\"os-R\'enyi random graph, in the quasi-normal regime where $n p_0$ is not too small as made explicit in \eqref{n-p0}. 
The quasi-Poisson regime, where $n p_0$ is smaller, requires different arguments and the application of somewhat different tests, and this will be detailed in a separate paper under preparation.

For the time being, in the quasi-normal regime, we learned the following.
In the moderately sparse setting --- $n\gg N^{2/3})$ for known $p_0$ and $n\gg N^{3/4}$ for unknown $p_0$ --- this detection boundary is achieved by polynomial-time tests.
In the sparser setting, there is a large discrepancy between the information theoretic boundaries and performances of known polynomial tests, which in view of the Planted Clique Problem, is not surprising.  

It is of great interest to study this optimal detection boundary, this time under computational constraints, a theme of contemporary importance in statistics, machine learning and computer science.  This promisingly rich line of research is well beyond the scope of the present paper.

\section{Proofs}
\label{sec:proofs}

\subsection{Auxiliary results}

The following is Chernoff's bound for the binomial distribution.
Remember the definition of $H$ in \eqref{H}.  

\begin{lem}[Chernoff's bound] \label{lem:chernoff}
For any positive integer $n$, any $q, p_0 \in (0,1)$, we have
\beq \label{chernoff}
\pr{\Bin(n, p_0) \ge q n} \leq \exp\left(- n H(q) \right).
\eeq
\end{lem} 
A consequence of Chernoff's bound is Bernstein's inequality for the binomial distribution.
\begin{lem}[Bernstein's inequality]\label{lem:bernstein}
For positive integer $n$, any $p_0\in (0,1)$ and any $x>0$, we have
\[\mathbb{P}\left[\Bin(n,p_0)\geq np_0+x \right]\leq \exp\left[-\frac{x^2}{2[np_0(1-p_0)+x/3]}\right].\]
\end{lem}

We will need the following basic properties of the entropy function.

\begin{lem} \label{lem:H}
For $p_0 \in (0,1)$, $H(q)$ is convex in $q \in [0,1]$.  Moreover, 
\beq \label{H-lem}
H_{p}(q) = \begin{cases}
\frac{(q - p)^2}{2 p (1-p)} + O\big(\frac{(q - p)^3}{p^2}\big), & \frac{q}{p} \to 1; \\
p \big(r \log r - r + 1\big), & \frac{q}{p} \to r \in (1, \infty), \ p \to 0; \\
q \log\big(\frac{q}{p}\big) + O(q), & \frac{q}{p} \to \infty.
\end{cases}
\eeq
\end{lem}

We will also use the following upper bound on the binomial coefficients.

\begin{lem} \label{lem:binom}
For any integers $1 \le k \le n$,
\beq \label{binom}
k \log(n/k) \le \log {n \choose k} \le k \log (ne/k),
\eeq
where $e = \exp(1)$.
\end{lem}

The next result bounds the hypergeometric distribution with the corresponding binomial distribution.  Let ${\rm Hyp}(N, m, n)$ denotes the hypergeometric distribution counting the number of red balls in $n$ draws from an urn containing $m$ red balls out of $N$.

\begin{lem} \label{lem:hyper}
${\rm Hyp}(N, m, n)$ is stochastically smaller than ${\rm Bin}(n, m/(N-m))$.
\end{lem} 

\begin{proof}
  Suppose the balls are picked one by one without replacement.  At
  each stage, the probability of selecting a red ball is smaller than
  $m/(N-m)$.  The result follows.
\end{proof}

\subsection{Proof of \thmref{clique}}
\label{sec:clique-proof}

Following standard lines, we start by reducing the composite alternative to a simple alternative by considering the uniform prior $\pi$ on subsets $S \subset [N] := \{1, \dots, N\}$ of size $|S| = n$.  The resulting likelihood ratio is
\beq
L = \frac{\# \{S \subset [N]: |S| = n, W_S = \nn\}}{{N \choose n} p_0^{n(n-1)/2}},
\eeq
which is the observed number of cliques of size $n$ divided by the expected number under the null.  

The risk of any test for the original problem is well-known to be bounded from below by the risk of the likelihood ratio test $\{L > 1\}$ for this `averaged' problem, which is equal to
\[
\gamma_L := \P_0(L > 1) + \E_0(L \{L \le 1\}).
\]
Therefore, it suffices to show that $\gamma_L \to 1$.  Here we use arguably the simplest method, a second moment argument, which is based on the fact that
\[
\gamma_L = 1 - \E_0 |L - 1| \ge 1 - \sqrt{\Var_0(L)},
\]
by the Cauchy-Schwarz inequality, so that it is enough to prove that $\Var_0(L) \to 0$.  We do so by showing that $\E_0 (L^2) \le 1 + o(1)$.  

Note that 
\[
L = p_0^{- \nn} \pi\left[W_S = \nn\right],
\]
where $\pi[\cdot]$ denotes the expectation with respect to $\pi$.  Hence, 
by Fubini's theorem, we have
\[
\E_0 L^2 = \pi^{\otimes 2}\left[p_0^{- 2\nn} \P_0(W_{S_1} = W_{S_2} = \nn)\right] = \pi^{\otimes 2}\left[p_0^{-K(K-1)/2}\right],
\]
where $K := |S_1 \cap S_2|$.  Indeed, the event $\{W_{S_1} = W_{S_2} = \nn\}$ means that all edges between pairs of nodes in $S_1$ exist, and similarly for $S_2$, and there are a total of $n(n-1) + K(K-1)/2$ such edges. 

Before going further, note that \eqref{clique1} and \eqref{binom} imply that
\beq \label{clique1-log}
\log(N/n) - \frac{(n-1)}2 \log(1/p_0) \to \infty.
\eeq
In particular, this means that $n \le 3 \log N$, eventually, and therefore
\beq \label{rho}
\frac{n^2}{N} = O((\log N)^2/N) \to 0.
\eeq

Since $K \sim {\rm Hyp}(N, n, n)$, by \lemref{hyper}, $K$ is stochastically bounded by $\Bin(n, \rho)$, where $\rho := n/(N-n)$.  Hence, with and \lemref{chernoff}, we have
\begin{eqnarray}
\P(K \ge k) 
&\le& \P({\rm Hyp}(N, n, n) \ge k) \nonumber \\
&\le& \P(\Bin(n, \rho) \ge k) \nonumber \\
&\le& \exp\left(- n H_{\rho}(k/n) \right) \label{Kbound}.
\end{eqnarray}
Now, using \lemref{H} and \eqref{rho}, for $k \ge 2$ we get
\[
n H_{\rho}(k/n) = k \log(k/(n \rho)) + O(k) = k \log (k N/n^2) + O(k).
\]

Hence,
\begin{eqnarray}
\pi^{\otimes 2}\left[p_0^{-K(K-1)/2}\right] 
&=& \P_0(K \le 1) + \sum_{k=2}^n \exp\left(\frac{k(k-1)}2 \log(1/p_0) -n H_{\rho}(k/n)\right) \nonumber \\
&\le& 1 + \sum_{k=2}^n \exp\left(k \left[\frac{(k-1)}2 \log(1/p_0) - \log (k N/n^2) + O(1)\right] \right) \label{sum1}.
\end{eqnarray}

For $a > 0$ fixed, the function $x \to a x - \log x$ is decreasing on $(0, 1/a)$ and increasing on $(1/a, \infty)$.  Therefore,
\[
\frac{(k-1)}2 \log(1/p_0) - \log(k N/n^2) \le -\omega,
\]
where
\[ 
\omega := \min \left(\log(N/n^2) - \frac12 \log(1/p_0), \ \log(N/n) -\frac{n-1}2 \log(1/p_0)\right).
\]
By \eqref{clique1-log}, the second term in the maximum tends to $\infty$.  This also the case of the first term, since 
\[
\log(N/n^2) -\frac12 \log(1/p_0) = \log(N/n) -\frac{n-1}2 \log(1/p_0) + \frac{n}2 \log(1/p_0) -\log n,
\]
with the second difference bounded from below.  Hence, $\omega \to\infty$.  Hence, the sum in \eqref{sum1} is bounded by 
\[
\sum_{k=2}^n \exp\left(k [\omega + O(1)]\right) \le \frac{e^{-\omega/2}}{1 -e^{-\omega/2}} \to 0,
\]
eventually.

Hence we showed that $\E_0 (L^2) \le 1 + o(1)$ and the proof of \thmref{clique} is complete.

\subsection{Proof of \thmref{lower}}
\label{sec:lower-proof}

We assume that \eqref{n-p0}, \eqref{lower1} and \eqref{lower2} hold.  We reduce the composite alternative to a simple alternative by considering the uniform prior $\pi$ on subsets $S \subset [N] := \{1, \dots, N\}$ of size $|S| = n$.  The resulting likelihood ratio is
\beq \label{L}
L(A) = {N \choose n}^{-1} \sum_{|S| = n} L_S(A) = \pi\big[ L_S(A) \big], 
\eeq
where $\pi[ \cdot ]$ is the expectation with respect to $S \sim \pi$, $A = (W_{i,j} : 1 \le i < j \le N)$ and 
\beq \label{Ldef}
L_S := \exp(\theta W_S - \Lambda(\theta) \nn),
\eeq
with
\beq \label{theta-def}
\theta := \theta_{p_1}, \quad \theta_q := \log\left(\frac{q (1-p_0)}{p_0 (1-q)}\right)
\eeq
and
\[
\Lambda(\theta) := \log(1 -p_0 +p_0 e^\theta),
\]
which is the moment generating function of ${\rm Bern}(p_0)$. 

Still leaving $p_0$ implicit, let $H_{p_0}(q)$ be short for $H(q)$.  It is well-known that $H$ is the Fenchel-Legendre transform of $\Lambda$; more specifically, for $q \in (p_0,1)$, 
\beq \label{HL}
H(q) = \sup_{\theta \ge 0} [q \theta - \Lambda(\theta)] = q \theta_q - \Lambda(\theta_q).
\eeq 
The second moment argument used in \secref{clique-proof} is also applicable here, though it does not yield sharp bounds.  In Case 1 below (see Subsection \ref{sec:proof_kmin}), which is the regime where the moderate deviations of the binomial come into play, this method leads to a requirement that the limit superior in \eqref{lower2} be bounded by $1/2$ instead of 1.  And, worse than that, in Case 3 below, which is the regime where the large deviations of the binomial are involved, it does not provide any useful bound whatsoever. 

Fortunately, a finer approach was suggested by \cite{MR1456646}.  The refinement is based on bounding the first and second moments of a truncated likelihood ratio.  Here we follow \cite{1109.0898}.  They work with the following truncated likelihood
\[
\Lt = {N \choose n}^{-1} \sum_{|S| = n} \1_{\Gamma_S} \, L_S\ .
\]
where the events $\Gamma_S$ will be  specified below. We note $\Gamma = \bigcap_{|S| = n} \Gamma_S$. 
Using the triangle inequality, the fact that $\Lt\leq L$ and the Cauchy-Schwarz inequality, we have the following upper bound:
\beqn
\E_0|L-1| 
&\le& \E_0|\Lt-1|+\E_0 (L-\Lt) \\
&\le& \sqrt{\E_0[\Lt^2] - 1 + 2 (1 - \E_0[\Lt])} + (1-E_0[\Lt]) \ ,
\eeqn
so that $\gamma_L\to 1$ when $\E_0[\Lt^2]\to 1$ and $\E_0[\Lt] \to 1$. Note that contrary to  \cite{1109.0898}, we do not require that $\P_0(\Gamma)\to 1$.
 More precisely, we shall prove that $(1,1)$ is an accumulation point of any subsequence of $(\E_{0}\Lt, \E_0[\Lt^2])$.
Adopting this approach allows us to assume that 
$p_1/p_0$ converges to $r\in [1,\infty]$, $p_1^2/p_0$ converges to $r_2\in [0,\infty]$ and that  
\beq \label{lower-2}
\frac{n H(p_1)}{2 \log (N/n)} < 1 - \eta_0,
\eeq
for some $\eta_0 \in (0,1)$ fixed.  Notice that \eqref{n-log} and \eqref{lower2} imply that $H(p_1) \to 0$, which by \lemref{H} forces either $p_1/p_0 \to 1$ or $p_1 \to 0$; in any case, $p_1$ is bounded away from 1 this time.

In what follows, we provide the general arguments while the proof of the technical results (Lemmas \ref{lem:kmin}-\ref{lem:ent_p1_qk}) is postponed to the end of the section. To show these technical results, we  divide the analysis depending on the behaviour of $p_1/p_0$
\begin{numcases}{\frac{p_1}{p_0} \to }
r = 1, \label{case1}\\
r \in (1,\infty), \label{case2}\\
r = \infty. \label{case3} 
\end{numcases}
In regime \eqref{case1}, the moderate deviations of the binomial distribution dominate and these are asymptotically equivalent to normal (Gaussian) deviations; in particular, it is in this setting (with $p_0$ constant) that \cite{1109.0898} successfully reduce the binary setting to the normal setting.  In regime \eqref{case3}, the large deviations of the binomial distribution dominate, which are not alike the normal deviations and lead to a completely different regime.  Regime \eqref{case2} is intermediary and requires special treatment.

First, we need some notations to introduce $\Gamma_S$.
Define the numbers 
\begin{eqnarray}
 k_*&=& \left[1+ 2\, \frac{\log(N/n)}{\log\left(1+\frac{(p_1-p_0)^2}{p_0(1-p_0)}\right)}\right]\wedge n\label{eq:def_k*} \ ,\\
k_{\min}& = &  \left[1+ 2\, \frac{\log\left(\frac{Nk_{*}}{n^2}\right)-\log\left\{\log\left(\frac{n}{\log(N/n)}\right)\wedge \log(N/n)\right\}}{\log\left(1+\frac{(p_1-p_0)^2}{p_0(1-p_0)}\right)}\right] \wedge n\label{eq:def_kmin}\ .
\end{eqnarray}
The exact expression of $k_{\min}$  will be useful for bounding the second moment of $\Lt$. For the time being, we only need to have in mind the properties summarized in the following lemma. 
\begin{lem}\label{lem:kmin}
We have $k_{\min}\rightarrow \infty$, $k_{\min}\sim k_*$,  and $\log(n/k_{\min})=o\left[\log(N/n)\right]$.
\end{lem}
 We define $\Gamma_S $ as follows
\beq \label{GammaS1}
\Gamma_S := \bigcap_{k=\lfloor \kmin\rfloor +1 }^n \{W_T \leq w_k, \ \forall T \subset S \text{ such that } |T| =k\} \ ,
\eeq
where  $w_k := q_k \kk$, with
\beq \label{qk}
\frac{(k-1)}2 H(q_k) =  \log(N/k)+2\ .
\eeq
This construction is possible by the following lemma, which serves as a definition.
\begin{lem}\label{lem:qk}
For any integer $k$ between $k_{\min}+1$ and $n$, there exists a unique $q_k\in (p_0,1)$ such that
\begin{eqnarray*}
\frac{(k-1)}2 H(q_k) =  \log(N/k)+2\ .
\end{eqnarray*}
Moreover, $q_k$ satisfies $\theta_{q_k}\leq 2\theta$.
\end{lem}

\subsubsection{First truncated moment}   
We first prove that $\E_0 \Lt \to 1$.  By Fubini's theorem, we have
\[
\E_0 \Lt = \pi \big[\E_0 [L_S \1_{\Gamma_S}] \big ] = \pi\big[\P_S(\Gamma_S)\big] = \P_S(\Gamma_S),
\]
where $S$ is any fixed subset of size $n$ in $\{1, \dots, N\}$ and this last inequality is by the fact that $\P_S(\Gamma_S)$ does not depend on $S$ by symmetry.  
By the union bound, Chernoff's bound \eqref{chernoff} and \eqref{binom},
\beqn
1 - \P_S(\Gamma_S) 
&\le& \sum_{k = \lfloor \kmin\rfloor +1}^n \ \sum_{T \subset S, |T|=k} \P_S(W_T > q_k \kk) \\
&\le& \sum_{k = \lfloor \kmin\rfloor +1}^n {n \choose k} \P\big(\Bin(\kk, p_1) > q_k \kk\big) \\
&\le& \sum_{k = \lfloor \kmin\rfloor +1}^n \exp\left[k \left(\log(ne/k) - \frac{(k-1)}2 H_{p_1}(q_k) \right)\right] \ .
\eeqn
We then conclude that $1 - \P_S(\Gamma_S) = o(1)$ using the following result.

\begin{lem}\label{lem:ent_p1_qk}
We have
\beq\label{eq:ent_p1_qk}
\min_{k= \lfloor k_{\min}\rfloor +1,\ldots,n}\left(\frac{k-1}{2}H_{p_1}(q_k)-\log\left(\frac{n}{k}\right)\right)\rightarrow \infty\ .
\eeq 
\end{lem}

\subsubsection{Second truncated moment}
We now prove that $\E_0 \Lt^2 \le 1 + o(1)$, which with $\E_0 \Lt \to 1$ shows that $\Var_0(\Lt) \to 0$.  
Let $S_1, S_2 \iid \pi$ and define $K = |S_1 \cap S_2|$.  By Fubini's theorem, we have
\beqn
\E_0 \Lt^2 
&=& \E_{S_1,S_2} \E_0 \left(L_{S_1} L_{S_2} \1_{\Gamma_{S_1}} \1_{\Gamma_{S_2}} \right) \\
&=& \pi^{\otimes 2}\big[ \E_0 \left(\exp\left(\theta (W_{S_1} + W_{S_2}) - 2 \Lambda(\theta) \nn\right) \1_{\Gamma_{S_1} \cap \Gamma_{S_2}} \right) \big]\ .
\eeqn
Define 
\[
W_{S \times T} = \frac12 \sum_{i \in S, j \in T} W_{i,j}\ ,
\]
and note that $W_S = W_{S \times S}$.
We use the decomposition
\beq \label{decomp}
W_{S_1} + W_{S_2} = W_{S_1 \times (S_1 \setminus S_2)} + W_{S_2 \times (S_2 \setminus S_1)} + 2 W_{S_1 \cap S_2}\ ,
\eeq
the fact that
\[
\Gamma_{S_1} \cap \Gamma_{S_2} \subset \{W_{S_1 \cap S_2} \le w_K\}\ ,
\]
and the independence of the random variables on the RHS of \eqref{decomp}, to get
\[
\E_0 \left(\exp\left(\theta (W_{S_1} + W_{S_2}) - 2 \Lambda(\theta) \nn\right) \1_{\Gamma_{S_1} \cap \Gamma_{S_2}} \right) \le {\rm I} \cdot {\rm II} \cdot {\rm III}\ ,
\]
where
\[
{\rm I} := \E_0 \exp\left(\theta W_{S_1 \times (S_1 \setminus S_2)} - \frac{\Lambda(\theta)}2 (n-K)(n+K-1)\right) = 1\ , 
\]
\[
{\rm II} := \E_0 \exp\left(\theta W_{S_2 \times (S_2 \setminus S_1)} - \frac{\Lambda(\theta)}2 (n-K)(n+K-1)\right) = 1\ , 
\]
\[
{\rm III} := \E_0 \left(\exp\left(2 \theta W_{S_1 \cap S_2} - 2 \Lambda(\theta) K^{(2)} \right) \IND{W_{S_1 \cap S_2} \le w_K}\right)\ . 
\]
The first two equalities are due to the fact that the likelihood integrates to one. 

To bound ${\rm III}$, we follow \cite{1109.0898}, with a twist.  When $K \leq  \kmin$, we will use the obvious bound:
\beqn
{\rm III} 
& \le& \E_0 \exp\left(2 \theta W_{S_1 \cap S_2} - 2 \Lambda(\theta) K^{(2)} \right) = \exp\left(\Delta K^{(2)}\right), 
\eeqn
where
\beq \label{Delta}
\Delta := \Lambda(2 \theta) - 2 \Lambda(\theta) = \log\left(1 + \frac{(p_1 -p_0)^2}{p_0 (1 -p_0)}\right).
\eeq

When $K > \kmin$, we use a different bound.  For any $\xi \in (0, 2 \theta)$, we have
\beqn
{\rm III} 
&\le& \E_0 \left[\exp\left(\xi W_{S_1 \cap S_2} + (2\theta -\xi) w_K - 2 \Lambda(\theta) K^{(2)} \right) \{W_{S_1 \cap S_2} \le w_K\}\right] \\
&\le& \E_0 \exp\left[\xi W_{S_1 \cap S_2} + (2\theta -\xi) w_K - 2 \Lambda(\theta) K^{(2)} \right]\ ,
\eeqn
 so that
\[ 
{\rm III} \le \exp\left(\Delta_K K^{(2)}\right),
\] 
where
\beq \label{Deltak}
\Delta_k := \min_{\xi \in [0, 2 \theta]} \Lambda(\xi) + (2\theta -\xi) q_k - 2\Lambda(\theta)\ .
\eeq
By the variational definition of the entropy \eqref{HL}, the minimum of $\Lambda(\xi) + (2\theta -\xi) q_k - 2\Lambda(\theta)$ over $\xi$ in $\mathbb{R}^+$ is achieved at $\xi=\theta_{q_k}$, and we know from Lemma \ref{lem:qk} that $\theta_{q_k}\leq 2\theta$. Hence, we have
\begin{eqnarray}
\Delta_k&=& - H(q_k)+2\theta q_k-2\Lambda(\theta)\nonumber \\
& = & -2H_{p_1}(q_k)+ H(q_k)\ ,\label{eq:delta_k}
\end{eqnarray}

Following our tracks, we have
\[
\E_0 \Lt^2 \le \E \left[\IND{K \leq  \kmin} \exp\left(\Delta K^{(2)}\right)\right] + \E \left[\IND{K > \kmin}\exp\left(\Delta_K K^{(2)}\right)\right],
\]
where the expectation is with respect to $\pi^{\otimes 2}$.

Let $b$ be an integer sequence such that $b \to \infty$ so slowly that 
\beq \label{b}
\frac{(p_1 - p_0)}{\sqrt{p_0}} \, \frac{b n^2}{N} \to 0,
\eeq
which is possible because of \eqref{lower1}.
Recall that $\rho = n/(N-n)$ and define $k_0 = \lceil b n \rho \rceil$.  We divide the expectation into two parts: $K \le k_0$ and $k_0 + 1 \le K \le n$.  When $k_0 = 1$, we simply have
\[
\E\left[ \IND{K \le k_0} \exp\left(\Delta K^{(2)}\right)\right] = \P(K \le 1) \le 1\ .
\]
When $k_0 \ge 2$, we use the expression \eqref{Delta} of $\Delta$ to derive
\beqn
\E \left[\IND{K \le k_0} \exp\left(\Delta K^{(2)}\right)\right] &\leq  &\exp\left[\Delta k_0^2\right]\\
&\leq& \exp\left[O(1)\frac{(p_1-p_0)^2}{p_0(1-p_0)}\frac{b^2 n^2}{N^2}\right] = 1+o(1)
\eeqn
because of \eqref{b}.

When $k_0 + 1 \le K \leq \lfloor \kmin\rfloor $, we use the bound \eqref{Kbound} and the identity $(1-x)\log(1-x)\geq -x$, to get
\beqn
\E \left[\IND{k_0 + 1 \le K \le \lfloor \kmin\rfloor} \exp\left(\Delta K^{(2)}\right) \right]
&\le &\sum_{k = k_0 + 1}^{\lfloor \kmin\rfloor} \exp\left[\Delta \frac{k(k-1)}2 - nH_{\rho}\left(\frac{k}{n}\right)\right]\\
&\leq & \sum_{k = k_0 + 1}^{\lfloor \kmin\rfloor} \exp\left[k\left(\Delta\frac{k-1}{2}-\log\left(\frac{k}{n\rho}\right)+1\right)\right]
\eeqn
For $a > 0$ fixed, the function $f(x) = a x - \log x$ is decreasing on $(0, 1/a)$ and increasing on $(1/a, \infty)$.  Therefore, for $k_0 + 1 \le k \le n$,
\[
\Delta \frac{k-1}2 - \log\left(\frac{k}{n\rho}\right) \le -\omega\ ,
\]
where
\[
\omega := \min \left[\log b - \Delta \frac{k_0-1}2, \ \log\left(\frac{\kmin}{n\rho}\right) - \Delta \frac{\kmin-1}{2}\right]\ .
\]
From what we did previously, we know that $\Delta (k_0-1) = o(1)$, so that the first term in the maximum tends to $\infty$.  Therefore, it suffices to look at the second term in the maximum.  In fact, $k_{\min}$ has been precisely defined in \eqref{eq:def_kmin} to make this second term diverge. Indeed, by \eqref{eq:def_kmin} and \eqref{Delta}, we have
\[
\Delta \frac{\kmin-1}{2}\leq \log\left(\frac{Nk^*}{n^2}\right)-\log\log\left[\frac{n}{\log(N/n)}\right]\ . 
\]
By Lemma \ref{lem:kmin} and since $\rho \asymp n/N=o(1)$, we get $\log(k_{\min}/(n\rho))- \log\left(\frac{Nk^*}{n^2}\right)=o(1)$. Consequently, 
\[
\log\left(\frac{\kmin}{n\rho}\right) - \Delta \frac{\kmin-1}2 \ge  \log\log\left[\frac{n}{\log(N/n)}\right]+o(1) \to \infty \ ,
\]
because of \eqref{n-log}.

When $K > \kmin$, we have 
\[
\E \left[\IND{K > \kmin} \exp\left(\Delta_K K^{(2)}\right) \right]
\le \sum_{k = \lfloor \kmin\rfloor +1}^{n} \exp\left[k\left(\Delta_k \frac{k-1}2 - \log\left(\frac{k}{n\rho}\right)+1\right)\right].
\]
Now, using \eqref{eq:delta_k},  we have
\[
\Delta_k \frac{k-1}2 - \log\left(\frac{k}{n\rho}\right)= \frac{k-1}{2}\left[-2H_{p_1}(q_k)+H(q_k)\right]- \log\left(\frac{N}{k}\right)+ 2\log\left(\frac{n}{k}\right)+o(1)\ ,
\]
which goes to $-\infty$ uniformly over all $k$ between $\lfloor k_{\min}\rfloor +1$ and $n$ by the definition \eqref{qk} of $q_k$ and by the control of $H_{p_1}(q_k)$ from Lemma \ref{lem:ent_p1_qk}. Hence,  the sum above tends to zero.  

This concludes the proof that  $\E_0 \Lt^2 \le 1 + o(1)$.

\subsubsection{Proof of Lemma \ref{lem:kmin}}\label{sec:proof_kmin}
We only need to prove that $k_*\rightarrow \infty$ and that $\log(n/k_*)=o\left[\log(N/n)\right]$ since \[\log\left\{\log\left(\frac{n}{\log(N/n)}\right)\wedge \log(N/n)\right\}=o(\log(N/n))\ .\]

We divide the analysis into three cases depending on the behaviour of $p_1/p_0$.

\medskip \noindent {\bf CASE 1:} $p_1/p_0\rightarrow 1$. 
Then, Lemma \ref{lem:H} tells us  $\log\left(1+\frac{(p_1-p_0)^2}{p_0(1-p_0)}\right)\sim 2H(p_1)$, so that 
\[k^*\succ \frac{\log(N/n)}{H(p_1)}\wedge n\succ n\ , \]
since $H(p_1)<2(1-\eta_0)\log(N/n)/n$ by \eqref{lower-2}. 
Hence $k^*\to \infty$ and $\log(n/k^*)=O(1)$. 

\medskip \noindent {\bf CASE 2:} $p_1/p_0\to r$ with $r\in (1,\infty)$. 
Since $H(p_1)$ goes to $0$, this  enforces $p_0\to 0$. Using Lemma~\ref{lem:H} and \eqref{lower-2}, we derive that 
\[p_0\left[r\log(r)-r+1\right] \prec \log(N/n)/n\ .\]
 Hence, $\log(N/n)/p_0 \succ n$. 
Going back to the definition of $k_*$, we derive that
\[k^*\succ \left[1+ \frac{\log(N/n)}{p_0(r-1)^2}\right]\wedge n \succ n\ .\]

\medskip \noindent {\bf CASE 3:} $p_1/p_0\to \infty.$ 
Again, we have $p_0\to 0$. By Lemma \ref{lem:H} and \eqref{lower-2}, 
\begin{equation}\label{eq:upper_bound_p1}
 p_1\log\left(\frac{p_1}{p_0}\right)\prec\frac{\log(N/n)}{n}\ .
\end{equation}
 Hence, 
\[\log\left(\frac{p_1}{p_0}\right) \prec \log\left[\log(N/n)/(np_0)\right] = o[\log(N/n)],\]
where the last part comes from  \eqref{n-p0}.
Hence,
\[k^*\succ \frac{\log(N/n)}{\log(p_1/p_0)} \to \infty\ .\]
Since \eqref{eq:upper_bound_p1} also implies that $p_1\prec \log(N/n)/n$, we have 
\begin{eqnarray*}
 \frac{n}{k^*}\prec \frac{n\log(1+p_1^2/p_0)}{\log(N/n)}\vee 1 \prec\frac{np_1^2}{\log(N/n)p_0}\vee 1\prec \frac{\log(N/n)}{np_0}\vee 1\ ,
\end{eqnarray*}
so that $\log(n/k^*)\leq  \log\left[\log(N/n)/(np_0)\right]\vee 0 + O(1)  = o[\log(N/n)]$ by \eqref{n-p0}.

\subsubsection{Proof of Lemma \ref{lem:qk}}\label{sec:proof_qk}
\def\qt{\widetilde{q}}
Define $\qt$ by the equation 
\begin{eqnarray}\label{eq:def_qt}
\frac{\qt}{1-\qt}= \frac{p_1^2(1-p_0)}{p_0(1-p_1)^2}\ ,
\end{eqnarray}
which implies $\theta_{\qt}=2\theta$. 
Because $H$ is strictly increasing and continuous on $(p_0, \qt)$, to prove the existence of $q_k$ it suffices to show that
\[\frac{\kmin-1}{2}H(\qt)\geq \log(N/\kmin)+2\ .\]

As in the proof of the previous lemma, we consider different cases depending on the convergence of $p_1/p_0$  and of $p_1^2/p_0$. 
In all cases, except the last one, we show that 
\[k_* H(\qt) \ge 2 (1+\eps) \log(N/n),\]
for some fixed $\eps > 0$, which suffices by \lemref{kmin}.
If $k_* < n$, so that $k_* \ge \frac2\Delta \log(N/n)$ (with $\Delta$ defined in \eqref{Delta}). If $k^*=n$ and $k_{\min}<n$, we have $k_* \ge \frac2\Delta \log(N/n)(1+o(1))$. Hence, it is enough the prove that 
\[H(\qt) \ge (1+\eps) \Delta, \quad \text{ for some fixed $\eps > 0$}.\]

The last case, Case 3(c) below --- which corresponds to $p_0=o(\log(N/n)/n)$ and $\log(n)=o(\log(N))$ --- requires a more delicate treatment.

\medskip \noindent {\bf CASE 1:} $p_1/p_0\rightarrow 1.$ 
By the definition of $\qt$, we have $\qt-p_0= (p_1-p_0)\left[1+\frac{p_1(1-p_1)}{p_0-2p_0p_1+p_1^2}\right]\sim 2(p_1-p_0)$ 
and Lemma \ref{lem:H} tells us that 
\[H(\qt)\sim \frac{2(p_1-p_0)^2}{p_0(1-p_0)} \ge 2 \Delta\ .\]   

\medskip \noindent {\bf CASE 2:} $p_1/p_0\to r$ with $r\in (1,\infty).$ 
Note that this forces $p_1\rightarrow 0$. 
Here \eqref{eq:def_qt} implies that $\qt/p_0 \sim (p_1/p_0)^2$, so that $H(\qt)\sim p_0\left(r^2\log(r^2)-r^2+1\right)$ by \lemref{H}.
At the same time, $\Delta \sim p_0 (r - 1)^2$, so that
\[\frac{H(\qt)}{\Delta} \sim \frac{r^2\log(r^2)-r^2+1}{(r-1)^2} = 1 + \frac{2r\left(r\log(r)-r+1\right)}{(r-1)^2} > 1\ .\]

\medskip \noindent {\bf CASE 3(a):} $p_1/p_0\to \infty$ and $p_1^2/p_0 \to 0.$ 
We have $\qt/p_0 \sim (p_1/p_0)^2 \to \infty$, implying that $H(\qt)\sim \qt \log(\qt/p_0) \sim 2 (p_1^2/p_0) \log(p_1/p_0)$ by \lemref{H}.  Also, $\Delta \sim \log(1+p_1^2/p_0)\sim \frac{p_1^2}{p_0}$.  Hence, $H(\qt) \gg \Delta$.

\bigskip \noindent {\bf CASE 3(b):} $p_1/p_0\to \infty$ and $p_1^2/p_0\to r_2\in (0,\infty).$  
Here $\qt \to 1/(1+r_2)$, so that $\qt/p_0 \to \infty$, implying that $H(\qt) \sim \qt \log(\qt/p_0) \asymp \log(1/p_0) \to \infty$.  Also, $\Delta \to \log(1+r_2)$.  Hence, $H(\qt) \gg \Delta$.

\bigskip\noindent {\bf CASE 3(c):} $p_1^2/p_0\to \infty$. By Definition \eqref{eq:def_k*} of $k_*$, this implies $k_*<n$.
 By definition of $\qt$, we have $\qt = 1 - o(1)$, so that $H(\qt) \sim \log(1/p_0)$.  On the other hand, $\Delta \sim \log(p_1^2/p_0)$.  Therefore, 
\[\frac{H(\qt)}{\Delta} \sim \frac{\log(1/p_0)}{\log(p_1^2/p_0)} = \frac1{1 - \frac{\log (p_1^2)}{\log(p_0)}},\]
so that we are done if $\log(p_1)/\log(p_0)$ is bounded away from 0. 
When $\log(p_1)/\log(p_0) = o(1)$, we need to work a little harder and perform a second order analysis.  From the definition of $\qt$, we derive $1-\qt\leq \frac{p_0}{p_1^2}$, so that 
\[H(\qt) \ge H(1 - \frac{p_0}{p_1^2}) = (1-\frac{p_0}{p_1^2}) \log(\frac{1-\frac{p_0}{p_1^2}}{p_0}) + \frac{p_0}{p_1^2} \log(\frac{\frac{p_0}{p_1^2}}{1-p_0}) = (1-\frac{p_0}{p_1^2}) \log(\frac1{p_0}) + o(1).\] 
Hence,
\begin{eqnarray*}
\frac{H(\qt)}{\Delta} -1 
&\ge& \frac{\log(1/p_1^2) -\frac{p_0}{p_1^2}\log(1/p_0)-o(1)}{ \log\left(\frac{p_1^2}{p_0}\right)+o(1)}\\
&\ge& \frac{2 \log(1/p_1)}{\log(1/p_0)}\left(\frac{1 - \frac{p_0 \log(p_0)}{p_1^2 \log(p_1^2)} + o(1)}{1 - \frac{2\log(p_1)}{\log(p_0)} + o(1)} \right) \\
&=& (2 + o(1))\frac{\log(1/p_1)}{\log(1/p_0)} \ .
\end{eqnarray*}
since $p_1^2/p_0 \to \infty$.
We use this lower bound to get
\begin{eqnarray*}
\frac{\kmin-1}{2} H(\qt)
&\ge&  \big[\log\left(N/k_*\right)- 2\log\left(n/k_*\right)-\log\log(n/\log(N/n))\big] \frac{H(\qt)}{\Delta}\\
&\geq & \big[\log(N/\kmin)+2\big] \times \left[1 - \frac{2 + o(1) + 2\log(n/k_*) + \log\log(n/\log(N/n)) }{\log(N/n)} \right]\\
&& \qquad \times \left[1+ (2+o(1))\frac{\log(1/p_1)}{\log(1/p_0)}\right]\ ,
\end{eqnarray*}
where we used \lemref{kmin} in the second inequality.
In order to conclude, because of \eqref{n-log},
it suffices to show that
\beq\label{eq:case3-c}
\frac{\log(n/k_*) + \log\log(n/\log(N/n))}{\log(N/n)} \ll \frac{\log(1/p_1)}{\log(1/p_0)}\ .
\eeq
The bound \eqref{eq:upper_bound_p1}, coupled with $p_1 \gg \sqrt{p_0}$, implies that $2\log\log(N/n)-\log(n)+\log(1/(np_0))\to \infty$. This, together with \eqref{n-p0}, forces $\log(n)= o\left[\log(N/n)\right]$.
Hence,
\[\frac{\log(1/p_0)}{\log(N/n)} = \frac{\log(n)+\log(1/(np_0))}{\log(N/n)}= o(1)\ .\]
It remains to show that 
\[\frac{\log(n/k_{*}) + \log\log(n/\log(N/n))}{\log(1/p_1)} = O(1).\]
By definition of $k_*$
\[\log(n/k_{*}) \le \log(n/\log(N/n)) + \log(\Delta) \le \log(n/\log(N/n)) + \log \log(p_1^2/p_0),\]
so that, because of \eqref{n-log} and \eqref{eq:upper_bound_p1}, we have  
\beqn
\frac{\log(n/k_{*}) + \log\log(n/\log(N/n))}{\log(1/p_1)} 
&\prec& \frac{\log(n/\log(N/n)) + \log \log(p_1^2/p_0) + \log \log(n/\log(N/n))}{\log(n/\log(N/n)) + \log \log(p_1/p_0)} \\
&=& O(1)\ .
\eeqn

\subsubsection{Proof of Lemma \ref{lem:ent_p1_qk}}\label{sec:proof_ent_p1_qk}

We first note that, by the entropy bound \eqref{lower-2} involving $p_1$, the definition of $q_k$ \lemref{qk}, definition of $\tilde{q}$ in \eqref{eq:def_qt}, and the fact that $H(q)$ is strictly increasing over $q > p_0$, we have
\beq \label{p1-qk-qt}
p_1 \le q_k \le \tilde q, \quad \forall k \le n \ .
\eeq

\noindent {\bf CASE 1:} $p_1/p_0\to 1$. In the proof of Lemma \ref{lem:qk} (Case 1), we have shown that $\tilde{q}$ defined in \eqref{eq:def_qt} satisfies $\tilde{q}\sim p_0$. 
By \eqref{p1-qk-qt}, we then get $q_k\sim p_0\sim p_1$. Then using \lemref{H} and the bound on the entropy \eqref{lower-2}, we get
\beq \label{H3}
\frac{(q_k - p_0)^2}{(p_1 - p_0)^2}\sim\frac{H(q_k)}{H(p_1)} \geq  \frac{n}{(1 -\eta_0)k}  \ge \frac{1}{1-\eta_0}\ .
\eeq
Hence, we may lower bound  $H_{p_1}(q_k)$as follows:
\begin{eqnarray*}
 H_{p_1}(q_k)\sim \frac{(q_k-p_1)^2}{2p_1(1-p_1)}\sim \frac{(q_k-p_0)^2}{2p_0(1-p_0)}\left(1-\frac{p_1-p_0}{q_k-p_0}\right)^2\succ H(q_k)[1-\sqrt{1-\eta_0}]^2\ ,
\end{eqnarray*}
which allows us to conclude that
\begin{eqnarray*}
 \frac{(k-1)}{2}H_{p_1}(q_k)\succ \frac{(k-1)}{2}H(q_k)\succ \log(N/n)\gg \log\left(\frac{n}{k}\right)\vee 1\ ,
\end{eqnarray*}
where the last inequality follows from Lemma \ref{lem:kmin} and the fact that $k \ge \kmin$.  

\medskip \noindent 
{\bf CASE 2:} $p_1/p_0\to r\in (1,\infty)$. As in the proof of \lemref{qk} (Case 2), we have $p_1 \to 0$. In the proof of Lemma \ref{lem:qk} (Case 1), we have shown that  $\tilde{q}/p_0\rightarrow r^2$ and that $\tilde{q} \rightarrow 0$. 
By \eqref{p1-qk-qt}, we can use the second asymptotic expression of the entropies  in Lemma \ref{lem:H}.
The inequalities in \eqref{H3} still hold, giving
\begin{eqnarray}\label{eq_lower_H1_2}
\frac{1}{1-\eta_0} \leq\frac{H(q_k)}{H(p_1)}\sim\frac{\frac{q_k}{p_0}\log\left(\frac{q_k}{p_0}\right)-\frac{q_k}{p_0}+1}{r\log(r)-r+1} = \frac{f(q_k/p_0)}{f(r)}\ ,
\end{eqnarray}
where $f(x) := x\log(x)-x+1$.  Since $f$ is convex and satisfies $f'(x)=\log(x)$, we have $f(x)-f(r)\leq (x-r)\log(x)$ for $x \ge r \ge 1$. Taking $x=q_k/p_0$ and using \eqref{eq_lower_H1_2}, we derive that 
\[\log\left(\frac{q_k}{p_0}\right)\left(\frac{q_k}{p_0}-r\right)\geq f(r) \left(\frac{f(q_k/p_0)}{f(r)} -  1\right) \ge \frac{f(r)\eta_0}{1-\eta_0} (1+o(1)) \ge f(r) \eta_0\ , \]
eventually.
As a consequence, $q_k/p_0$ is also lower bounded away from $r$. Thus,  $\log(q_k/p_1)/\log(q_k/p_0)$ is bounded away from $0$ by a constant that only depends on $r$ and $\eta_0$.  We then derive,
\begin{equation}\label{eq:upper_bound_fr}
\log\left(\frac{q_k}{p_1}\right)\left(\frac{q_k}{p_1}-1\right)\succ \log\left(\frac{q_k}{p_0}\right)\left(\frac{q_k}{p_0}-r\right)\ .
 \end{equation}
Now, for the entropy $H_{p_1}(q_k)$, by \lemref{H} we have 
\beqn
H_{p_1}(q_k)&\succ& \frac{(q_k - p_1)^2}{p_1} \wedge q_k\log\left(\frac{q_k}{p_1}\right) \\
&=& p_1\left[\left(\frac{q_k}{p_1}-1\right)^2\wedge \frac{q_k}{p_1}\log\left(\frac{q_k}{p_1}\right)\right]\geq p_1\left(\frac{q_k}{p_1}-1\right) \log\left(\frac{q_k}{p_1}\right)
\eeqn
as $\log(1+x)\leq x$. Since $H(p_1)\sim p_0f(r)$, we get by \eqref{eq_lower_H1_2} and \eqref{eq:upper_bound_fr}
\beqn
H_{p_1}(q_k)
&\succ& \frac{rH(p_1)}{f(r)}\left(\frac{q_k}{p_1}-1\right) \log\left(\frac{q_k}{p_1}\right) \\
&\succ& \frac{H(q_k)}{f(q_k/p_0)} \left(\frac{q_k}{p_0}-r\right) \log\left(\frac{q_k}{p_0}\right) \\
&\succ & H(q_k) \\
&\succ& \frac1k \log(N/n)\ ,
\eeqn
where the third line follows from the fact that the $q_k/p_0$ is lower bounded away from $r$ and that $f(x)\sim x\log(x)$ when $x\rightarrow \infty$.
Thus, 
\[ \frac{k-1}{2}H_{p_1}(q_k) \succ \log(N/n) \gg \log(n/k)\vee 1\ ,
\]
as before.

\medskip \noindent 
{\bf CASE 3:} $p_1/p_0\rightarrow \infty$. As in the proof of \lemref{qk} (Case 2), we have $p_1 \to 0$. We start as in the two previous cases, again using \lemref{H} to get the asymptotic expressions of the entropies. By \eqref{p1-qk-qt}, $q_k/p_0 \geq p_1/p_0 \to \infty$, so that
\beq \label{ent_p1_qk_case3_1}
\frac{n}{(1-\eta_0)(k-1)}\leq\frac{H(q_k)}{H(p_1)}\sim  \frac{q_k}{p_1}\frac{\log(q_k/p_0)}{\log(p_1/p_0)}\sim  \frac{q_k}{p_1}\left[1+ \frac{\log(q_k/p_1)}{\log(p_1/p_0)}\right]
\eeq
It follows that $\frac{q_k}{p_1}(1+  \frac{\log(q_k/p_1)}{\log(p_1/p_0)})\geq (1-\eta_0)^{-1}$. Since $\log(p_1/p_0)\rightarrow \infty$, we derive that $q_k/p_1\geq (1-\eta_0/2)^{-1}$ for $n$ large enough.
Since $p_1\leq q_k\leq \tilde{q}$, we have $q_k/p_1 \le \tilde q/p_1 \le p_1/p_0$.
It follows that $\log(q_k/p_1)/\log(p_1/p_0) \le 1$, and therefore $q_k/p_1 \ge (1+o(1)) \frac{n}{2k}$ by \eqref{ent_p1_qk_case3_1}. 
We conclude that
\beq \label{ent_p1_qk_case3_2}
\frac{q_k}{p_1} \ge \left[\frac{n}{2k}\vee \frac{1}{1-\eta_0/2}\right] (1+o(1))\ .
\eeq

Turning to the entropy $H_{p_1}(q_k)$, we have $H_{p_1}(q_k)\geq q_k\log(q_k/p_1)-q_k+(1-q_k)p_1$. Using
\lemref{qk} and \lemref{H}, we get
\begin{eqnarray*}
 \frac{k-1}{2}H_{p_1}(q_k) \geq \frac{H_{p_1}(q_k)}{H_{p_0}(q_k)}\log\left(\frac{N}{k}\right)\geq \frac{\log\left(\frac{q_k}{p_1}\right)-1+\frac{p_1}{q_k} -p_1}{\log\left(\frac{q_k}{p_0}\right)}\log\left(\frac{N}{n}\right)(1+o(1))\ .
\end{eqnarray*}
We explaing above that $q_k/p_1 \le \tilde q/p_1 \le p_1/p_0$, so that $q_k/p_0 \le (p_1/p_0)^2$, implying $\log(q_k/p_0)\leq 2\log(p_1/p_0)$.
Applying \eqref{ent_p1_qk_case3_2}, we get
\begin{eqnarray*}
 \frac{k-1}{2}H_{p_1}(q_k)\succ \frac{\log(N/n)}{\log(p_1/p_0)}\left[\log(n/k)\vee 1\right]\ ,
\end{eqnarray*}
We saw in the proof of Lemma \ref{lem:kmin} (Case 3) that $\log(p_1/p_0)=o[\log(N/n)]$, so  we conclude that 
\[\frac{k-1}{2}H_{p_1}(q_k)\gg \log(n/k)\vee 1\ .\]

\subsection{Proof of Theorem \ref{thm:lower_unknownp_0}}
\label{sec:lower-proof_unknown}

We start with a couple of lemmas.

\begin{lem}\label{lem:entropy_p0}
Under conditions \eqref{n-p0_unknown}, \eqref{lower1_eq_unknown} and \eqref{lower2_eq_unknown}, we have 
\beq \label{eq:hypo_p0_tildep0}
\limsup\frac{n H_{p_0}(p_1)}{2 \log (N/n)} < 1\ , \quad \quad\quad \quad  \frac{(p_1-p_0)^2}{p_0}\frac{n^3}{N^{3/2}}\rightarrow 0\ . 
\eeq
\end{lem}

As in the proof of \thmref{lower}, for $n$ large enough, we may assume that there exists $\eta_0>0$ such that 
\beq\label{eq:condition_entrop_unknown} 
\frac{n H_{p_0}(p_1)}{2 \log (N/n)} = 1-\eta_0\ .
\eeq

\begin{lem}\label{lemma:n3}
Under conditions \eqref{n-p0_unknown} and (\ref{eq:condition_entrop_unknown}), we have 
\[
 \frac{n^2}{N}\frac{(p_1-p_0)^2}{p_0}=o(1)\ .
\]
\end{lem}

We consider the likelihood ratio under the uniform prior:
\beq \label{L_unknown}
L' = {N \choose n}^{-1} \sum_{|S| = n} L'_S = \pi\big[ L'_S \big], 
\eeq
and 
\beq \label{Ldef_unknown}
L'_S := \exp\left[\theta_{p_1} W_S - \Lambda(\theta_{p_1}) \nn + \theta_{p'_0}(W-W_S) - (N^{(2)}-n^{(2)})\Lambda(\theta_{p'_0})\right]\ .
\eeq

As in the proof of \thmref{lower}, we use a thresholded version of $L'$ to prove that $\E_0[|L'-1|]=o(1)$:
\[
\Lt := {N \choose n}^{-1} \sum_{|S| = n} L'_S \1_{\Gamma_S}\ ,
\]
where $\Gamma_S$ is defined in \eqref{GammaS1}.
As in the proof of Theorem \ref{thm:lower}, we prove that any subsequence of $\E_0[\Lt-1]$ has $0$ as an accumulation point. This allows us to assume that $p_1/p_0$ converges to $r\in [1,\infty]$ and that $p^2_1/p_0$ converges to $r_2\in [0,\infty]$. To control $\E_0[\Lt-1]$, it suffices to prove that $\E_0 \Lt=1+o(1)$ and that $\E_0[\Lt^2]\leq 1+o(1)$.

\medskip
\noindent {\bf First moment}
\[
\E_0 \Lt = \pi \big[\E_0 [L'_S \1_{\Gamma_S}] \big] = \pi\big[\P'_S(\Gamma_S)\big] = \P'_S(\Gamma_S)\ .
\]
As the proof of \thmref{lower}, we can show that $\P'_S(\Gamma_S)=1+o(1)$ relying only on \eqref{lower2_eq_unknown}.

\medskip
\noindent {\bf Second Moment}.
It remains to prove that $\E_0[\Lt^2]\leq 1+o(1)$. Let $S_1, S_2 \iid \pi$ and define $K = |S_1 \cap S_2|$.  Observe that ($W_{S_1\cap S_2}$, $W_{S_1}+W_{S_2}-2W_{S_1\cap S_2}$, $W-W_{S_1}-W_{S_2}+W_{S_1\cap S_2}$) are independent.
Arguing as in the proof of \thmref{lower}, we decompose the square of the modified likelihood as follows. 
\beqn
\E_0 \Lt^2 
&=& \pi^{\otimes 2}\big[ \E_0 \big(L'_{S_1} L'_{S_2} \1_{\Gamma_{S_1}} \1_{\Gamma_{S_2}} \big) \big] \\
&\leq & \pi^{\otimes 2}\big[ {\rm I} \cdot {\rm II} \cdot {\rm III} \big] \, 
\eeqn
where
\beqn
{\rm I} &:=& \E_0 \exp\left[2\theta_{p'_0}(W-W_{S_1}-W_{S_2}+W_{S_1\cap S_2}) - 2\Lambda(\theta_{p'_0})\left(N^{(2)}-2n^{(2)}+K^{(2)}\right)\right]\ ,\\
{\rm II}& :=& \E_0 \exp\left[\left(\theta_{p_1}+\theta_{p'_0}\right) \left(W_{S_1}+W_{S_2}-2W_{S_1\cap S_2}\right) - 2\left(\Lambda(\theta_{p_1})+\Lambda(\theta_{p'_0}) \right)(n^{(2)}-K^{(2)})\right]\ , \\
{\rm III}& :=& \E_0 \left[\exp\left(2 \theta_{p_1} W_{S_1 \cap S_2} - 2 \Lambda(\theta_{p_1}) K^{(2)} \right) \{W_{S_1 \cap S_2} \le w_K\}\right]\ .
\eeqn
All these expectations only depend on $S_1$ and $S_2$ through $K$. 

The term ${\rm III}$ already appeared in the proof of Theorem \ref{thm:lower}, where we saw that ${\rm III}\leq \exp\left(\Delta K^{(2)}\right)$ for $K\leq k_{\min}$, and that ${\rm III}\leq \exp\left(\Delta_K K^{(2)}\right)$ for $K> k_{\min}$ where $k_{\min}$ is defined in \eqref{eq:def_kmin}, while $\Delta$ and $\Delta_k$ are defined in \eqref{Delta} and \eqref{Deltak}, respectively. 

Since the expectations inside ${\rm I}$ and ${\rm II}$ are not thresholded, we easily compute these terms:
\[{\rm I} = \exp\left[\left(N^{(2)}-2n^{(2)}+K^{(2)}\right) \big(\Lambda(2\theta_{p'_0}) - 2\Lambda(\theta_{p'_0}) \big) \right]\ ,\]
with
\[\Lambda(2\theta_{p'_0}) - 2\Lambda(\theta_{p'_0}) = \log\left(1+\frac{(p'_0-p_0)^2}{p_0(1-p_0)}\right) \le \frac{(p_1-p'_0)^2}{{p_0(1-p_0)}}\left(\frac{n^{(2)}}{N^{(2)}}\right)^2\ ;\]
and
\[{\rm II} = \exp\left[2 \left(n^{(2)}-K^{(2)}\right) \big( \Lambda(\theta_{p_1} + \theta_{p'_0}) - \Lambda(\theta_{p_1}) -  \Lambda(\theta_{p'_0}) \big) \right]\ ,\]
with 
\[\Lambda(\theta_{p_1} + \theta_{p'_0}) - \Lambda(\theta_{p_1}) -  \Lambda(\theta_{p'_0})
= \log\left(1-\frac{(p_0-p'_0)(p_1-p_0)}{p_0(1-p_0)}\right) \le - \frac{(p_1-p'_0)(p_1-p_0)}{{p_0(1-p_0)}}\frac{n^{(2)}}{N^{(2)}}\ .\]
Since $(p_1-p'_0)= (p_1-p_0)(1-n^{(2)}/N^{(2)})^{-1}$, we derive
\beq\label{eq:upperI.II}
{\rm I}\cdot {\rm II} \leq \exp\left[\frac{(p_1-p_0)^2}{p_0(1-p_0)}\left(-\frac{(n^{(2)})^2}{N^{(2)}}+ \frac{n^{(2)}}{N^{(2)}}\left(K^{(2)}-\frac{(n^{(2)})^2}{N^{(2)}}\right)\frac{2-\frac{n^{(2)}}{N^{(2)}}}{\left(1-\frac{n^{(2)}}{N^{(2)}}\right)^2}\right)\right] =: V_K\ .
\eeq
By Lemma \ref{lemma:n3}, $\Delta n^2/N \rightarrow 0$ and by \eqref{eq:hypo_p0_tildep0},  $\Delta n^3/N^{3/2}\rightarrow 0$. Hence, there exists $b\rightarrow \infty$ such that $\Delta \frac{n^3}{N^{3/2}}b^2\rightarrow 0$ and $\Delta b\frac{n^2}{N} \to 0$.  Define $k'_0= \lfloor  \frac{n^2}{N}+ \frac{n}{N^{1/2}}b\rfloor$ and $k_0=\lfloor  b\frac{n^2}{N}\rfloor$. We can take $b$ small enough to constrain $k_0\leq n/2$.

\medskip
To prove that  $\E_0 \Lt^2\leq 1+o(1)$, we only need to  show the four following results 
\begin{eqnarray}\label{eq:terme_difficile_unknown}
\E \left[\{K \leq k'_0\} \exp\left\{\Delta K^{(2)}\right\} V_K\right]&\leq& 1+o(1)\ ,\\
\E \left[\{k'_0< K \le k_0\} \exp\left\{\Delta K^{(2)}\right\} V_K\right]&=&o(1)\ , \label{eq:terme2_difficile_unknown}\\
\E \left[\{k_0< K \leq \kmin\} \exp\left\{\Delta K^{(2)}\right\}V_K\right]&=& o(1)\ .\label{eq:terme3_difficile_unknown}\\
\E \left[\{\kmin <K \leq n\} \exp\left\{\Delta_K K^{(2)}\right\} V_K\right]&=& o(1)\ .\label{eq:terme4_difficile_unknown}
\end{eqnarray}
By Lemma \ref{lemma:n3} and the definition \eqref{eq:upperI.II} of $V_k$, we have $\log(V_k)=o(k^2/N) = o(k)$ when $k\leq n$. As a consequence, the expectations in \eqref{eq:terme3_difficile_unknown} and \eqref{eq:terme4_difficile_unknown} are almost the same as the expectations $\E \left[\{k_0< K \leq \kmin\} \exp\left\{\Delta K^{(2)}\right\}\right]$ and $\E \left[\{\kmin K \leq n\} \exp\left\{\Delta_K K^{(2)}\right\}\right]$ that we bounded in the proof of Theorem \ref{thm:lower}. 
This is made rigorous to establish the following result.
\begin{lem}\label{lem:Klarge}
 Under the entropy condition \eqref{eq:condition_entrop_unknown}, the bounds \eqref{eq:terme3_difficile_unknown} and \eqref{eq:terme4_difficile_unknown} hold. 
\end{lem}
In fact the main difference between the proof of \thmref{lower} and the current proof lies in the control of the two expectations in \eqref{eq:terme_difficile_unknown} and \eqref{eq:terme2_difficile_unknown}. Here, we need to carefully upper bound $V_K$ in order  to balance $\Delta K^{(2)}$. 
 Using the identity $\log(1+x)\leq x$, the property $\log(V_k)\leq 0$ for $k\leq n/2$ --- easily verified from the definition \eqref{eq:upperI.II} --- and $k_0\leq n/2$, we get 
\[V_k\leq \exp\left[\Delta \left(-\frac{(n^{(2)})^2}{N^{(2)}}+ \frac{n^{(2)}}{N^{(2)}}\left(k^{(2)}-\frac{(n^{(2)})^2}{N^{(2)}}\right)\frac{2-\frac{n^{(2)}}{N^{(2)}}}{\left(1-\frac{n^{(2)}}{N^{(2)}}\right)^2}\right)\right]\ ,\]
for $k\leq k_0$. In the sequel, we note  \[\Delta':= \frac{\Delta}{2}\left[1+ \frac{n^{(2)}}{N^{(2)}}\frac{2-\frac{n^{(2)}}{N^{(2)}}}{\left(1-\frac{n^{(2)}}{N^{(2)}}\right)^2}\right]\ , \]
so that $2\Delta'\sim \Delta$. Thus, we get for any $k\leq k_0$, 
\begin{eqnarray}
\exp\left\{\Delta k^{(2)}\right\}V_k&\leq& \exp\left\{2\Delta'\left(k^{(2)}-\frac{(n^{(2)})^2}{N^{(2)}}\right)\right\} \nonumber\\
&\leq &
\exp\left\{\Delta'\left(k^2-\frac{n^4}{N^2}\right)+ 2\Delta'\frac{n^3}{N^2}\right\}\nonumber\\
&\leq & (1+o(1))\exp\left\{\Delta'\left(k^2-\frac{n^4}{N^2}\right)\right\}\label{eq:upper_deltaV} \ ,
\end{eqnarray}
since $\Delta'n^3/N^2\prec \Delta n^3/N^2 \leq \frac{(p_1-p_0)^2}{p_0(1-p_0)}\frac{n^3}{N^2}=o(n/N)=o(1)$ by Lemma \ref{lemma:n3}.

\medskip
Using this upper bound \eqref{eq:upper_deltaV}, we consider the expectation in  \eqref{eq:terme_difficile_unknown} 
\beqn
\E \left[\{K \leq k'_0\} \exp\left\{\Delta K^{(2)}\right\} V_K\right]&\prec & \exp\left[\Delta'\left(k'^2_0-\frac{n^4}{N^2}\right)\right]\\
&\prec & \exp \left[\Delta'\left(k'_0-\frac{n^2}{N}\right)\left(k'_0+\frac{n^2}{N}\right)\right]\\
&\leq& \exp\left[\Delta' \frac{b n}{N^{1/2}} \big(\frac{2 n^2}N + \frac{b n}{N^{1/2}}\big) \right]= 1+o(1)
\eeqn
since $\Delta' \frac{b^2  n^2}{N}\prec \Delta \frac{b^2 n^2}{N}=o(1)$ and $\Delta' \frac{b n^3}{N^{3/2}} \ll \Delta \frac{b^2 n^3}{N^{3/2}} = o(1)$ by definition of $b$. We have proved \eqref{eq:terme_difficile_unknown}.

\medskip
To prove \eqref{eq:terme2_difficile_unknown}, we apply the Cauchy-Schwarz inequality and we upper bound $K$ by $k_0\leq bn^2/N$,
\beqn
\E \left[\{k'_0<K \leq  k_0\} \exp\left\{\Delta'\left(K^2- \frac{n^4}{N^2}\right)\right\}\right]&\leq& \E \left[\{k'_0<K \leq  k_0\} \exp\left\{\Delta'(b+1)\frac{n^2}{N}\left(K- \frac{n^2}{N}\right)\right\}\right]\\
&\leq & \P^{1/2}(K > k'_0) \ \E^{1/2}\left[\exp\left\{2\Delta'(b+1)\frac{n^2}{N}\left(K- \frac{n^2}{N}\right)\right\}\right]
\eeqn

Recall that $K \sim {\rm Hyp}(N, n, n)$, so that $\E K= \frac{n^2}N$ and $\Var(K) \le \frac{n^2}{N}$.  Hence, by Chebyshev's inequality, $\P(K> k'_0) \le 1/b^2 \to 0$.

We know from \citep[p.173]{aldous85} that $K$ has the same distribution as the random variable $\mathbb{E}(W|\mathcal{B}_p)$ where $W$ is binomial random variable of parameters $n$, $n/N$ and $\mathcal{B}_N$ some suitable $\sigma$-algebra. 
By a convexity argument, we apply this to get 
\beqn
\E\exp\left\{2\Delta'(b+1)\frac{n^2}{N}\left(K- \frac{n^2}{N}\right)\right\}&\leq&\left[1+\frac{n}{N}\left(e^{2\Delta'(b+1)\frac{n^2}{N}}-1\right) \right]^{n}e^{-2\Delta'(b+1)\frac{n^4}{N^2}}\\
&\leq & \exp\left[4\frac{n^6}{N^3}(b+1)^2\Delta^{'2}\right]\leq 1+o(1)\ ,
\eeqn
since $\Delta'b(\frac{n^2}{N}\vee \frac{n^3}{N^{3/2}}) = o(1)$ by definition of $b$. 
All in all, we have proved \eqref{eq:terme2_difficile_unknown}.

\subsubsection{Proof of Lemma \ref{lem:entropy_p0}}
The second convergence is a straightforward consequence of the definition of $p_0$, \eqref{lower1_eq_unknown} and \eqref{lower2_eq_unknown}, so that we focus on the first result.
Let us compute the difference between the two entropies $H_{p'_0}(p_1)$ and $H_{p_0}(p_1)$.
\beqn
H_{p'_0}(p_1)-H_{p_0}(p_1)&=& p_1\log\left(\frac{p_0}{p'_0}\right)+(1-p_1)\log\left(\frac{1-p_0}{1-p'_0}\right)\\
&\leq &\frac{n^{(2)}}{N^{(2)}}\left[p_1\frac{p_1-p'_0}{p'_0}-  (1-p_1)\frac{p_1-p'_0}{1- p'_0}\right]\\
&\leq &\frac{n^{(2)}}{N^{(2)}}\frac{(p_1-p'_0)^2}{p'_0(1-p'_0)}
\eeqn
Arguing as in the proof of Lemma \ref{lemma:n3}, we note that, under  conditions \eqref{n-p0_unknown} and (\ref{eq:condition_entrop_unknown}),
\[\frac{n^2}{N}\frac{(p_1-p'_0)^2}{p'_0(1-p'_0)}=o(1)\ , \]
so that $H_{p'_0}(p_1)-H_{p_0}(p_1)=o(1/N)=o\left(\log(N/n)/n\right)$, since $n\leq N$.

\subsubsection{Proof of Lemma \ref{lemma:n3}}

{\bf CASE 1:} $p_1/p_0\rightarrow 1$.
By condition \eqref{eq:condition_entrop_unknown},
\beqn
 \frac{n^2}{N}\frac{(p_1-p_0)^2}{p_0(1-p_0)}\sim 2H(p_1)\frac{n^2}{N}\prec \log\left(\frac{N}{n}\right)\frac{n}{N} =o(1)\ .
\eeqn

\noindent 
{\bf CASE 2:} $p_1/p_0\rightarrow c\in (1,\infty)$.
Similarly,
\beqn
 \frac{n}{N}\frac{(p_1-p_0)^2}{p_0(1-p_0)}&\prec& p_0(c-1)^2\frac{n^2}{N}\\
&\prec& H(p_1)\frac{n^2}{N}
\prec \log\left(\frac{N}{n}\right)\frac{n}{N} =o(1)\ .
\eeqn

\noindent 
{\bf CASE 3:} $p_1/p_0\rightarrow \infty$.
We have
\beqn
 \frac{n^2}{N}\frac{(p_1-p_0)^2}{p_0(1-p_0)}\sim  \frac{p_1^2}{p_0}\frac{n^2}{N}\ .
\eeqn
By  condition (\ref{eq:condition_entrop_unknown}) and $p_1\log(p_1/p_0)\sim H(p_1)\prec \frac1n \log(N/n)$. Dividing  this inequality by $p_0$ and then taking the logarithm leads to $\log(p_1/p_0) \prec \log\log(N/n) + \log(1/np_0) = o(\log(N/n))$ by \eqref{n-p0_unknown}. It follows that $p_1/p_0=o(\sqrt{N/n})$ and $p_1=o(\log(N/n)/n)$. All in all, we conclude that 
\beqn
  \frac{p_1^2}{p_0}\frac{n^2}{N}=o\left[\log(N/n)\sqrt{\frac{n}{N}}\right]=o(1)\ .
\eeqn

\subsubsection{Proof of Lemma \ref{lem:Klarge}}
 Let us first consider \eqref{eq:terme4_difficile_unknown}. Using the upper bound $\log(V_k) = o(k)$, we only have to prove that 
\[\E \left[\{K > \kmin\}\exp\left(\Delta_K K^{(2)}+ o(K)\right)\right]=o(1)\ .\]
 We have shown in the proof of Theorem \ref{thm:lower} (only using the entropy condition) that 
\[
\E \left[\{K \ge \kmin\}\exp\left(\Delta_K K^{(2)}\right)\right]\leq \sum_{k=\lfloor k_{\min}\rfloor+1}^n \exp\left[k\left(\Delta_k\frac{k-1}{2}-\log\left(\frac{k}{n\rho}\right)+ 1\right)\right]\]
tends to zero since all the terms $\Delta_k\frac{k-1}{2}-\log\left(\frac{k}{n\rho}\right)+ 1$ simultaneously go to $-\infty$ for $k= \lfloor k_{\min}\rfloor+1,\ldots, n$.  Consequently, 
\[\E \left[\{K > \kmin\}\exp\left(\Delta_K K^{(2)}+ o(K)\right)\right]\leq \sum_{k=\lfloor k_{\min}\rfloor+1}^n \exp\left[k\left(\Delta_k\frac{k-1}{2}-\log\left(\frac{k}{n\rho}\right)+ 1+o(1)\right)\right]\]
also tends to zero.\\

\medskip

Let us turn to \eqref{eq:terme3_difficile_unknown} following again the same arguments as in the proof of Theorem \ref{thm:lower}.
\beqn
\lefteqn{\E \left[\{k_0<K \leq   \lfloor \kmin\rfloor \} \exp\left\{\Delta K^{(2)}+o(K)\right\}\right]}&&\\&\leq &\sum_{k = k_0 + 1}^{\lfloor \kmin\rfloor} \exp\left[\Delta k^{(2)}+ o(k) - nH_{\rho}\left(\frac{k}{n}\right)\right]\\
&\leq & \sum_{k = k_0 + 1}^{\lfloor \kmin\rfloor} \exp\left[k\left\{\frac{\Delta}{2}( k-1)+o(1)-\log\left(\frac{k}{n\rho}\right)+1\right\}\right]\ .
\eeqn
 Hence, as in the previous proof, we only need to prove that 
\[
\omega := \min \left[\log b - \Delta k_0/2, \ \log\left(\frac{\kmin}{n\rho}\right) - \Delta\kmin/2\right]
\]
goes to $\infty$. By definition of $k_0$, we have $\Delta k_0=o(1)$, while we showed in the previous proof that $\log\left(\frac{\kmin}{n\rho}\right) - \Delta \kmin \to \infty$. 
With this, we conclude.

\subsection{Proof of \prpref{total}}
\label{sec:total-proof}

We start with a useful result for proving that a test is asymptotically powerful based on the first two moments of the corresponding test statistic.

\begin{lem} \label{lem:cheb}
Suppose that for testing $H_0$ versus $H_1$, a statistic $T$ satisfies
\beq \label{cheb}
\cheb_T := \frac{\E_1(T) - \E_0(T)}{\max\big(\sqrt{\Var_1(T)}, \sqrt{\Var_0(T)}\big)} \to \infty.
\eeq
Then there is a test based on $T$ that is asymptotically powerful.
\end{lem}

\begin{proof}
Consider the test that rejects when $T \ge \E_0(T) + \sqrt{\cheb_T \Var_0(T)}$.  By Chebyshev's inequality, the probability of type I error tends to zero:
\[
\P_0(T \ge \E_0(T) + \sqrt{\cheb_T \Var_0(T)}) \le \frac1{\cheb_T} \to 0.
\]
For the probability of type II error, we have
\[
\P_1(T \ge \E_0(T) + \sqrt{\cheb_T \Var_0(T)}) = \P_1\left(\frac{T -\E_1(T)}{\sqrt{\Var_1(T)}} \ge - \gamma\right) \ge 1 -\frac1{\gamma^2},
\]
where
\[
\gamma := \frac{\cheb_T \max\big(\sqrt{\Var_1(T)}, \sqrt{\Var_0(T)}\big) - \sqrt{\cheb_T \Var_0(T)}}{\sqrt{\Var_1(T)}} \to \infty.
\]
\end{proof}

We now apply \lemref{cheb} to the total degree test.  From \eqref{total-dist}, under the null, 
\[
\E_0(\tot) = \frac{N(N-1)}2 p_0, \quad \Var_0(\tot) = \frac{N(N-1)}2 p_0 (1-p_0), 
\]
while under the alternative,
\[
\E_1(\tot) = \frac{N(N-1)}2 p_0 + \frac{n(n-1)}2 (p_1 -p_0),
\]
and
\[
\Var_1(\tot) = \frac{N(N-1)}2 p_0 (1-p_0) + \frac{n(n-1)}2 [p_1 (1-p_1) - p_0 (1-p_0)]. 
\]

In any case, 
\[
\max\big(\Var_1(\tot), \Var_0(\tot)\big) \le \frac12 N^2 p_0 + \frac12 n^2 (p_1 - p_0).
\]
Recalling the definition of $\cheb_\tot$ in \eqref{cheb}, under \eqref{total} we have 
\[
\cheb_\tot \ge \frac{n(n-1) (p_1 -p_0)}{\sqrt{N^2 p_0 + n^2 (p_1 - p_0)}} \asymp \frac{n^2}N  \frac{p_1 -p_0}{\sqrt{p_0}} \to\infty.
\]
Therefore, the total degree test is powerful when \eqref{total} holds.

\subsection{Proof of \prpref{scan}}
\label{sec:scan-proof}

We use the union bound, Chernoff's bound \eqref{chernoff} and \eqref{binom} to get
\beqn
\P_0(\scan \ge a \nn ) 
&\le& {N \choose n} \exp\left(- \nn H(a) \right) \\
&\le& \exp\left(n \log(Ne/n) - \nn H(a)\right),
\eeqn
which goes to zero when
\beq \label{scan-1}
\log(N/n) - \frac{(n-1)}2 H(a) \to -\infty.
\eeq
Choose $a = \eta p_0 + (1-\eta) p_1$ with $\eta \in (0,1)$ fixed, sufficiently small that 
\[
\liminf \frac{n H(a)}{2 \log (N/n)} > 1.
\]
This is possible because of how $H$ varies, which is described in \lemref{H}.  We then consider the test that rejects when $\scan \ge a \nn$.  We just chose $a$ so that its level tends to zero.  
Under the alternative, let $S$ denote the community.  By definition, $\scan \ge W_S$, and since $W_S \sim \Bin(\nn, p_1)$ and $p_1 \nn \to \infty$, $W_S = p_1 \nn + O_P(\sqrt{p_1 \nn})$.  Therefore, the test is powerful when $p_1 -a \gg \sqrt{p_1 \nn}$.  Since $p_1 -a = \eta(p_1 - p_0)$ and $\eta > 0$ is constant, this is the same as $(p_1 -p_0) n^2 \gg \sqrt{p_1 n^2}$.  Now, if $p_1/p_0$ is bounded away from 1, this is true because $p_1 -p_0 \asymp p_1$ and $p_1 n^2 \to \infty$; while if $p_1/p_0 \to 1$, we use \lemref{H} and \eqref{scan} to get that $(p_1 -p_0)^2 n/p_0 \ge {\rm cst} \log (N/n)$, implying that 
$(p_1 -p_0) n^2/\sqrt{p_1 n^2} \sim (p_1 -p_0) n/p_0 \to \infty$.

\subsection{Proof of Proposition \ref{prp:degre_variance}}

The arguments are based on cumbersome, but pedestrian moment calculations.

{\bf Under the null.}
We first show that $V^*$ remains bounded under the null.
Rewrite $V$ as
\begin{eqnarray}
V &= & \frac{1}{N-2}\sum_{i=1}^N\left(W_{i\cdot}-(N-1)p_0\right)^2+(\hat{p}_0-p_0)^2(N-1)\left[-\frac{N(N-1)}{N-2}+\frac{N^{(2)}}{N^{(2)}-1}\right]\nonumber\\
&+&(\hat{p}_0-p_0)(N-1)\frac{N^{(2)}}{N^{(2)}-1}(-1+2p_0) - (N-1)\frac{N^{(2)}}{N^{(2)}-1}p_0(1-p_0)\ .\label{eq:expression_2_V}
\end{eqnarray}
Since $\E_0(\hat{p}_0-p_0)^2=(N^{(2)})^{-1}p_0(1-p_0)$ and $\E_0(W_{i\cdot}-(N-1)p_0)^2=(N-1)p_0(1-p_0)$, it follows that $\E_0 V=0$. 
For the variance, we have 
\beqn
\Var_0\left[(\hat{p}_0-p_0)^2\right]&\leq &\frac{2{p_0}^2(1-p_0)^2}{(N^{(2})^2}+\frac{p_0(1-p_0)}{(N^{(2)})^3}, \ \text{ and} \\
\Var_0\left[\sum_{i=1}^N\left(W_{i\cdot}-(N-1)p_0\right)^2\right]& =& 2N(N-1)\left[(N-3){p_0}^2(1-p_0)^2+p_0(1-p_0)[{p_0}^3+(1-p_0)^3]\right]\ .
\eeqn
Hence, we get 
\beqn
\Var_0(V)\prec N{p_0}^2+p_0\prec N{p_0}^2\ ,
\eeqn
since $p_0 \succ 1/N$.
Therefore, by Chebyshev's inequality, $V = O_P(\sqrt{N} p_0)$.
Under the null, $\NN \hat p_0 = W \sim \Bin(\NN, p_0)$, and because $\NN p_0 \to \infty$, we have $\hat p_0 \ge \frac12 p_0$ with probability tending to 1 as $N \to \infty$.
We conclude that, under the null, $V^* = O_P(1)$. 

{\bf Under the alternative.}
Turning to the alternative hypothesis, we shall prove that $V^*$ tends to infinity with high probability by showing that $\E'_1(V)\gg \sqrt{N}p_0 \vee \sqrt{\Var'_1(V)}$ since  $\sqrt{N}\widehat{p}_0=O_{\P'_1}(\sqrt{N}p_0)$. 
The expression \eqref{eq:expression_2_V} of $V$ still holds. 

By definition of $p_0 =p'_0+n^{(2)}/N^{(2)} (p_1 - p'_0)$, we have $\E'_1(\hat{p}_0)= p_0$. Furthermore, 
\beqn
\E'_1[(\hat{p}_0-p_0)^2]-\frac{1}{N^{(2)}}p_0(1-p_0)
&\sim &- 4\frac{n^2}{N^4}(p_1-p'_0)^2 \\
\E'_1\left[\sum_{i=1}^N\left(W_{i\cdot}-(N-1)p_0\right)^2\right]- N(N-1)p_0(1-p_0)&\sim&  n^3(p_1-p'_0)^2
\eeqn
Inputing this into \eqref{eq:expression_2_V}, we get
\beq\label{eq:equivalent_E1V}
\E'_1[V]\sim (p_1-p'_0)^2\frac{n^3}{N}\ .
\eeq
By \eqref{degre_variance}, $\E'_1[V]\gg \sqrt{N}p'_0$ and $\E'_1[V]\gg n^2/N^{3/2}(p_1-p'_0)$ and it follows that 
 $\E'_1[V]\gg \sqrt{N}p_0$. To conclude, we need to control the variance of $V$ under $\P_1'$. Tedious computations lead us to 
\beqn
\Var'_1\left[\hat{p}_0-p_0\right]&\prec& \frac{p_0}{N^2}, \\
\Var'_1\left[(\hat{p}_0-p_0)^2\right]&\prec& \frac{p^2_0}{N^4}+ \frac{p_0}{N^6}, \text{ and}\\
\Var'_1\left[\sum_{i=1}^N\left(W_{i\cdot}-(N-1)p_0\right)^2\right]&\prec &N^2p_0+N^3{p_0}^2+ n^3(p_1-p'_0)^2+ n^3Np_0(p_1-p'_0)^2+n^4(p_1-p'_0)^3\ ,
\eeqn
so that, using the fact that $p_0 \succ 1/N$, we get
\beqn
\Var'_1[V]\prec N{p_0}^2 + \frac{n^3}{N} p_0 (p_1-p'_0)^2+\frac{n^4}{N^2}(p_1-p'_0)^3\ .
\eeqn

We conclude that $\E'_1[V] \gg \sqrt{\Var'_1(V)}$ as soon as the following cconditions are met  
\[(p_1-p'_0)^2\frac{n^3}{N} \gg \sqrt{N}p_0, \quad (p_1-p'_0) \frac{n^{3/2}}{N^{1/2}} \gg \sqrt{p_0}, \quad n (p_1-p'_0)^{1/2} \gg 1.\] 
We already argued that the first one holds, while the second and third are easily seen to be implied by the first condition and the fact that $p_0 \succ 1/N$.

\subsection{Proof of \prpref{scan_estimate}}

It suffices to show that the scan test is asymptotically powerful for $\wh H_0$ versus $H_1$, where the model under $\wh H_0$ is $\bbG(N, \wh p_0)$.
In view of \prpref{scan}, it is therefore enough to prove that, under \eqref{scan_estimate1}, we have
\[
\liminf \frac{n H_{\wh p_0}(p_1)}{2 \log(N/n)} > 1.
\]

First note that $\wh p_0 = W/N^{(2)}$  is concentrated around its mean.  Indeed, we have 
\[\E [ N^{(2)} \wh p_0 ] = (N^{(2)}-n^{(2)}) p_0 + n^{(2)} p_1 = N^{(2)} p_0 + n^{(2)} (p_1 -p_0),\]
and
\[\Var [ N^{(2)} \wh p_0 ] = (N^{(2)}-n^{(2)}) p_0 (1-p_0) + n^{(2)} p_1 (1 -p_1) \le \E [ N^{(2)} \wh p_0 ] .\]
Hence, by Chebyshev's inequality,
\[\wh p_0 = p_0 + a + O_P\big( \frac1{N} \sqrt{p_0 + a}\big), \quad a := \frac{n^{(2)}}{N^{(2)}} (p_1 -p_0).\] 
Since $p_0\gg N^{-2}$, we have $\sqrt{p_0}/N=o(p_0)$. If $a\geq p_0$, then $p_0\gg N^{-2}$ impies that $\sqrt{a}/N=o(a)$. All in all, we get $\sqrt{p_0+a}/N=o(p_0+a)$ and
$\wh p_0 \sim_P p_0 + a$.
As in the previous proofs, we can assume that $p_1/p_0 \to r \in [1,\infty]$. In the three following case, we prove that
\[
\liminf \frac{n H_{\wh p_0}(p_1)}{2 \log(N/n)} > 1.
\]

\noindent 
{\bf CASE 1:} $p_1/p_0\rightarrow 1$. In that case, we have $a=o(p_0)$ and $\sqrt{p_0}/N=o(p_1-p_0)$ since 
\[\frac{(p_1-p_0)^2}{p_0}\succ H_{p_0}(p_1)\succ \frac{\log(N/n)}{n}\ .\]
Hence, $\wh p_0-p_0=o(p_1-p_0)$ and we conclude that 
\beqn
H_{\wh p_0}(p_1)&\sim_P& \frac{(p_1-\wh p_0)^2}{2p_0(1-p_0)}\geq \frac{(p_1-p_0)^2 -2(p_1-p_0)(\wh p_0-p_0)}{2p_0(1-p_0)}
\sim_P H_{ p_0}(p_1)\ .
\eeqn

\noindent
{\bf CASE 2:} $p_1/p_0\rightarrow r\in (1,\infty)$. Hence, $a=o(p_0)$ and $p_0\sim_P\wh p_0$. It follows that
\beqn
H_{\wh p_0}(p_1)&\sim_P& \wh p_0(r\log(r)-r+1)\sim_P H_{ p_0}(p_1)\ .
\eeqn

\noindent 
{\bf CASE 3:} $p_1/p_0\rightarrow \infty$. Since $\wh p_0\sim_P p_0+a$, we derive
\beqn
H_{\wh p_0}(p_1)&\sim_P& p_1\log\left(\frac{p_1}{p_0+a}\right)\geq  p_1\log\left(\frac{p_1}{2(p_0\vee a)}\right)\sim_P H_{p_0}(p_1)\wedge 2p_1\log\left(\frac{N}{n}\right)
\eeqn
It remains to prove that $\lim\inf np_1>1$ when $\lim\inf nH(p_1)/\log(N/n)>2$. Assume that $\lim\inf np_1\leq 1$ so that there exists a subsequence satisfying 
\[\lim np_1\leq 1\text{ and }\lim\inf np_1\frac{\log(p_1/p_0)}{\log(N/n)}> 2.\] It follows that $\lim\inf \log(1/(np_0))/\log(N/n)> 2$ and $\lim\sup N^2p_0/n\leq 1$, which contradicts the assumption of the proposition. 
\medskip

\subsection{Proof of \prpref{relaxed}}
We prove the result when $p_0$ is known.  The situation when $p_0$ is unknown can be dealt with in a similar way; see, for example, the proof of \prpref{scan_estimate}.
Let $\bB = \bW^2$.
We first lower bound ${\rm SDP}_n(\bW^2)$ from below under the alternative where $S$ is the anomalous subset of indices.
We have 
\[
{\rm SDP}_n(\bB) \ge \lambda_n^{\rm max}(\bB) \ge \lambda_n^{\rm max}(\bB_S) \ge \frac1n {\bf 1}_S^\top \bW^2 {\bf 1}_S = \frac1n \sum_{i,j\in S} \sum_{k=1}^N W_{ik} W_{kj}.
\]
We have
\begin{align*}
\mu_S 
&:= \frac1n \E_S \big( \sum_{i,j\in S} \sum_{k=1}^N W_{ik} W_{kj} \big) \\
&= [(n-1) p_1 + (N-n) p_0] + (n-1) \big[(n-2)p_1^2 + (N-n) p_0^2\big], \\
&= (N-1) p_0 + (n-1) (p_1-p_0) + (n-1) (N-2) p_0^2 + (n-1)(n-2) (p_1^2 - p_0^2),
\end{align*}
and, after some tedious but straightforward calculations,
\[
\sigma_S^2 := \frac1{n^2} \Var_S \big( \sum_{i,j\in S} \sum_{k=1}^N W_{ik} W_{kj} \big) = O\left( (N/n) p_0 (1-p_0) [1 + (n p_0)^2] + p_1 (1-p_1) [1 + (n p_1)^2] \right).
\]
By Chebyshev's inequality, under the alternative, ${\rm SDP}_n(\bB) \ge \mu_S - O_P(\sigma_S)$.
  
Under the null, we bound ${\rm SDP}_n(\bB)$ from above as \cite{berthet} do.
Specifically, they use a result of \cite{bach2010convex}, which says that
\[
{\rm SDP}_n(\bB) = \min_{\bU} \lambda^{\rm max}(\bB + \bU) + n |\bU|_\infty,
\]
where the minimum is over symmetric matrices $\bU = (U_{ij})$ and $|\bU|_\infty := \max_{i,j} |U_{ij}|$.
Similar to what \cite{berthet} do, we apply this identity to $\bU = (U_{ij})$ with $U_{ij} = -B_{ij} \IND{|B_{ij}| \le z}$, obtaining
\[
{\rm SDP}_n(\bB) \le \lambda^{\rm max}(\tau_z(\bB)) + n z.
\]
where $\tau_z(\bB)$ is the hard thresholding of $\bB$ at threshold $z$, meaning the matrix with $(i,j)$ coefficient equal to $B_{ij} \IND{|B_{ij}| > z}$.
Under the null, we have
\[
B_{ii} = \sum_{k} W_{ik} \sim \Bin(N-1, p_0),
\]
and, when $i \ne j$,
\[
B_{ij} = \sum_{k} W_{ik} W_{jk} \sim \Bin(N-2, p_0^2).
\]

Fix $\eps > 0$. 
Using Bernstein's inequality (\lemref{bernstein}) and the union bound, we find that the following inequalities happen simultaneously with probability tending to one under the null:
\[
\max_{i} B_{ii} \le (N-1) p_0 + x_0, \quad x_0 := \sqrt{(1+\eps) 2 N p_0 (1-p_0) \log N} + (1+ \eps)\log(N) ,
\]
\[
\max_{i \ne j} B_{ij} \le (N-2) p_0^2 + x_{00}, \quad x_{00}:=  2 \sqrt{(1+\eps)  N p_0^2 (1-p_0^2) \log N} +  2(1+ \eps)\log(N) .
\]
Hence, choosing $z = (N-2) p_0^2 + x_{00}$, we have
\[
{\rm SDP}_n(\bB) \le \zeta := (N-1) p_0 + x_0 + n (N-2) p_0^2 + n x_{00},
\]
with high probability under the null. In order to conclude, we need to prove that $\mu_S - O(\sigma_S) > \zeta$ with probability going to one.

Before proceeding, we note that \eqref{relaxed} implies that, for some $\eta > 0$, 
\[
n p_1^2 > 2 (1 + \eta) p_0 \sqrt{N \log N}\ ,
\]
and \eqref{n-p0} implies that either $n p _0 \ge 1$, or $(N/n)^{-a} < n p_0 < 1$, for some sequence $a \to 0$. In particular, this implies
\[n^2\geq np_0\sqrt{N\log(N)}\geq \sqrt{N}(n/N)^a\  ,\]
so that $n\geq N^{1/4-a}$. It also follows that $n \sqrt{p_0} > \sqrt{n} (N/n)^{-a} \to \infty$.  

We have $\mu_S - \zeta \ge (1 + o(1)) n^2 p_1^2 - N p_0^2 - x_0 - n x_{00}$, with
\[\frac{n^2 p_1^2}{N p_0^2} > \frac{2 (1 + \eta) n p_0 \sqrt{ N \log N}}{N p_0^2} \asymp \frac{n \sqrt{\log N}}{\sqrt{N} p_0} > \frac{n^2 \sqrt{\log N} (N/n)^{a}}{\sqrt{N}} \geq \sqrt{\log N} \to \infty\ ,\]
\[\frac{x_0}{n x_{00}} \le \frac1{n \sqrt{p_0}} + \frac1n \to 0\ ,\]
\[\frac{n x_{00}}{n^2 p_1^2} \le \frac{2 n \sqrt{(1+\eps)  N p_0^2 \log N} + 2(1+ \eps)n \log(N)}{2 (1 + \eta) n p_0 \sqrt{ N \log N}} = \frac{1+\eps}{1+\eta} + \sqrt{\frac{(1+\eps)\log N}{(1+\eta)N p_0^2}} = \frac{1+\eps}{1+\eta} + o(1)\ , \]
since $ N p_0^2 > N n^{-2} (N/n)^{-2a} > N^{2t - 2a}$ with $  2t - 2a \to 2t > 0$.
Assuming that $\eta >  \eps$, it remains to show that $n^2 p_1^2 \gg \sigma_S$ to prove that $\mu_S - O(\sigma_S) > \zeta$ with probability going to one in the asymptote.

We have $\sigma_S^2 \asymp Np_0/n + N n p_0^3 + n^2 p_1^3$, and
\[\frac{n^2 p_1^2}{\sqrt{Np_0/n}} \succ n \sqrt{p_0} \sqrt{n \log N} \to \infty\ ,\]
since $n \sqrt{p_0} \to \infty$, and also
\[\frac{n^2 p_1^2}{\sqrt{N n p_0^3}} \succ \sqrt{n \log(N)/p_0} \to \infty\ ,  \]
and
\[\frac{n^2 p_1^2}{\sqrt{n^2 p_1^3}} = n \sqrt{p_1} \ge n p_1 \to \infty\ .\]

\subsection{Proof of Proposition \ref{prp:max}}
The first results follows from a simple consequence of Bernstein's inequality for binomial random variables and the union bound.  Details are omitted. Let us concentrate on the second bound. It suffices to prove that with probability $\mathbb{P}_S$ going to one $\max_{i\in S}W_{i\cdot}< \max_{i\in S^c}W_{i\cdot}$ since the distribution of $\max_{i\in S^c}W_{i\cdot}$ under $\mathbb{P}_S$ is stochastically smaller than the distribution of $\max_{i=1,\ldots, N}W_{i\cdot}$ under $\mathbb{P}_0$. 
Since $\lim\sup \log(n)/\log(N) <1$, we can assume that $n<N^{1-\epsilon}$ for some $\epsilon>0$. Condition \eqref{n-p0} also enforces $p_0\gg \log(N)/N$. Since the power of the maximal degree test is increasing with respect to $p_1$, we can assume that $p_1$  satisfies Condition \eqref{eq:condition_lower_degree} but is still large enough so that $p_1\gg \log(n)/n$.

Fix $\delta>0$ arbitrarily small. Applying Berntein's inequality (Lemma \ref{lem:bernstein}) and using $Np_0\gg \log(N)$ and $np_1\gg \log(n)$, we derive that
\begin{eqnarray}
\lefteqn{\max_{i\in S}W_{i\cdot}- (N-1)p_0+ (n-1)(p_1-p_0)}&&\nonumber \\& \leq & \sqrt{2(1+\delta)(N-1)p_0(1-p_0)\log(n)}\nonumber + \sqrt{(2+\delta)np_1(1-p_1)\log(n)} \nonumber\\
&\leq & \sqrt{2(1+\delta)(1-\epsilon)(N-1)p_0(1-p_0)\log(n)}(1+o(1))\ .\label{eq:upper_bound_Wi}
\end{eqnarray}
with probability going to one since we assume that $n(p_1-p_0)=o(\sqrt{N\log(N)p_0})=o(Np_0)$.

Let us consider a consider a subset $T\subset S^c$ of size $N^{1-\kappa}$ with some $\kappa>0$. As the $W_{i\cdot}$ are not independent, it is not straightforward to directly lower bound their supremum. This is why we compare it to independent variables. Let us call the $i^*_T$ the smallest $i$ in $T$ that achieves $\max_{i\in T}\sum_{j\in T^c}W_{i,j}$
\beqn
\max_{i\in S^c}W_{i\cdot}&\geq& \max_{i\in T}W_{i\cdot}
\geq  \max_{i\in T}\sum_{j\in T^c}W_{i,j} + \sum_{j\in T}W_{i^*_T,j}\ .
\eeqn
Observe that that the first term is supremum of $|T|$ independent binomial variables  and that the second term follows a binomial distribution with parameters $p$ and $|T|$.  
With probability going to one, we have $\sum_{j\in T}W_{i^*_T,j}\geq |T|p_0- \sqrt{|T|p_0(1-p_0)\log(|T|)}$.
Let us turn to the supremum of independent binomial distributions.
We start from $\mathbb{P}(\Bin(n,p)= k)= p^k(1-p)^{n-k}\binom{n}{k}$. Consider $p$ bounded away from $1$ and $k\geq np$ such that $k/n$ is also bounded away from one. 
Using the stirling formula $\sqrt{2\pi n}(n/e)^n<n!< \sqrt{2\pi n}(n/e)^ne^{1/(12n)}$, we get
\beqn
\mathbb{P}(\Bin(n,p)= k+i)&\geq &\exp\left[-nH_{p}(k/n)\right]\frac{1}{e^2\sqrt{2\pi k}}\left(\frac{p(1-k/n)}{k/n(1-p)}\right)^{i} \left(1-\frac{i}{k+i}\right)^{i}\ ,\\
\mathbb{P}(\Bin(n,p)\geq  k)&\succ& \exp\left[-nH_{p}(k/n)\right]\frac{1}{\sqrt{k}}\frac{k/n(1-p)}{k/n-p}\left[1-\left(\frac{p(1-k/n)}{k/n(1-p)}\right)^{\sqrt{k}}\right]\ ,
\eeqn
where we have summed the first inequality for $i=0,\ldots, \sqrt{k}-1$. 
Applying this lower bound to $\sum_{j\in T^c}W_{i,j}$ and using Lemma \ref{lem:H}, we derive that 
\beqn
\mathbb{P}_S\left[\sum_{j\in T^c}W_{i,j}\geq (N-1-|T|)p_0+  \sqrt{2(1-\delta)(N-|T|-1)p_0(1-p_0)\log(|T|)}\right]
&\succ & \frac{|T|^{1-\delta}}{\sqrt{(1-\delta)\log(|T|)}}\ .
\eeqn
Since the random variables $\sum_{j\in T^c}W_{i,j}$ for $i\in T$ are independent, it follows that 
\[\sup_{i\in T}\sum_{j\in T^c}W_{i,j}\geq (N-1-|T|)p_0+ \sqrt{2(1-\delta)(N-|T|-1)p_0(1-p_0)\log(|T|)}\ ,\]
with probability going to one. 
All in all, we derive that with probability going to one
\begin{eqnarray*}
\max_{i\in S^c}W_{i\cdot}\geq (N-1)p_0+ \sqrt{2(1-\delta)(N-1)p_0(1-p_0)(1-\kappa)\log(N)}-2\sqrt{2N^{1-\kappa}p_0(1-p_0)\log(N)}\ ,
\end{eqnarray*}
where the last term is negligible in front of the second term. Comparing this last lower bound with (\ref{eq:upper_bound_Wi}) and taking $\kappa$ and $\delta$ small enough allows us to conclude.

\subsection{Proof of Proposition \ref{prp:densest_subgraph}}

By Chebyshev's inequality $h(\cV)\sim_{\mathbb{P}_0} Np_0/2$ with probability going to one. Since $h(S)\leq |S|/2$, we have $h(S)\leq Np_0/4$ for all subsets $S$ of size smaller than $Np_0/2\rightarrow \infty$.
Note that $|S|h(S) \sim \Bin(|S|^{(2)}, p_0)$. Applying Bernstein inequality (Lemma \ref{lem:bernstein}) and Lemma \ref{lem:binom} to all  subsets $S$ of size larger than $Np_0/2$, we derive that 
\beqn
|S|h(S)\leq \frac{|S|^2}{2}p_0+ \sqrt{|S|^3p_0(1-p_0)\log\left(\frac{Ne^2}{|S|}\right)}+ |S|\log\left(\frac{Ne^2}{|S|}\right)
\eeqn	
with probability larger than $1- \exp(-Np_0/2)$. Comparing $h(S)$ with $Np_0/2$, we get
\beqn
\frac{2h(S)}{Np_0}\leq \frac{|S|}{N}+ 2\sqrt{\frac{|S|}{N}}\sqrt{\frac{2+\log(N/|S|)}{Np_0}} + 2\frac{2+ \log(N/|S|)}{Np_0}= \frac{|S|}{N}+ 2\sqrt{\frac{|S|}{N}}o(1)+o(1)\ ,
\eeqn
since $Np_0\gg \log(N)$. This quantity is away from one, except if $|S|\sim N$. As a consequence, $\max_S h(S)\sim_{\mathbb{P}_0} h(\cV)\sim_{\mathbb{P}_0} Np_0/2$ with probability going to one.\\

Let us turn to the alternative distribution. Under $\mathbb{P}_S$, $|S|h(S)\sim \Bin(|S|^{(2)}, p_1)$. It follows that $h(S)\sim_{\mathbb{P}_S} np_1/2$ with probability going to one. The densest subgraph test is therefore powerful when $\lim\inf \frac{np_1}{Np_0}>1$.\\

Let us now assume that  $\frac{np_1}{Np_0}\rightarrow 0$.
For any subset $T$, $|T|h(T)$ is the sum of two independent binomial distributions of parameters $(|S\cap T|^{(2)},p_1)$ and $(|T|^{(2)}-|S\cap T|^{(2)},p_0)$. Applying, as previously, Bernstein's inequality for all subsets $T$ of size larger than $Np_0/2$, we derive that
\beqn
|T|h(T)&\leq& \frac{|T|^{2}p_0}{2}+  \frac{|S\cap T|^2}{2}(p_1-p_0)+ \sqrt{|T|^{3}p_0\log\left(\frac{Ne^2}{|T|}\right)}+ |T|\log\left(\frac{Ne^2}{|T|}\right) \\&&+ \sqrt{2|T\cap S|^{3}p_1\log\left(\frac{ne}{|S\cap T|\vee 1}\right)}+ 2|S\cap T|\log\left(\frac{ne}{|S\cap T|\vee 1}\right)
\eeqn
with probability going to one. Comparing $h(T)$ with $Np_0/2$ we get
\beqn
\frac{2h(T)}{Np_0}\leq \frac{|T|}{N}+ \frac{n}{N}\frac{p_1-p_0}{p_0}+ o(1)+ \sqrt{\frac{np_1}{Np_0}}o(1)\ .
\eeqn
Since we assume that $np_1=o(Np_0)$, this quantity is away from one except if $|T|\sim N$.

\subsection*{Acknowledgements}

We would like to thank Jacques Verstraete for helpful discussions on the clique number of a random graph.  We first learned about the work of \cite{1109.0898} at the {\em New Trends in Mathematical Statistics} conference held at the Centre International de Rencontres Math\'ematiques (CIRM), Luminy, France, in 2011.  
The research of E. Arias-Castro is partially supported by a grant from the Office of Naval Research (N00014-13-1-0257). The research of N. Verzelen is partly supported by the french Agence Nationale
de la Recherche (ANR 2011 BS01 010 01 projet Calibration).

\bibliographystyle{chicago}
\bibliography{subgraph-detection}

\end{document}